\documentclass[11pt,letterpaper,reqno]{amsart}

\usepackage{amssymb}
\usepackage{amsmath}
\usepackage{amsthm}
\usepackage{bbm}
\usepackage{enumerate}
\usepackage{doi}

\usepackage{bookmark}
\usepackage{hyperref}
\hypersetup{pdfstartview={FitH}}

\newtheorem{theorem}{Theorem}[section]
\newtheorem{corollary}[theorem]{Corollary}

\newtheorem{lemma}[theorem]{Lemma}
\newtheorem{proposition}[theorem]{Proposition}
\theoremstyle{definition}
\newtheorem{definition}[theorem]{Definition}
\newtheorem{remark}[theorem]{Remark}
\newtheorem{example}[theorem]{Example}

\numberwithin{equation}{section}
\allowdisplaybreaks

\begin{document}

\title{On transposed Poisson conformal algebras}

\author{Lamei Yuan}
\address{Lamei Yuan (corresponding author), School of Mathematics, Harbin Institute of Technology, Harbin 150001, People's Republic of China}
\email{lmyuan@hit.edu.cn}

\author{Hao Fang}
\address{Hao Fang, School of Mathematics, Harbin Institute of Technology, Harbin 150001, People's Republic of China}
\email{23s012028@stu.hit.edu.cn}


\subjclass[2020]{17B69, 17B63, 17B61, 17D25}

\keywords{(noncommutative) transposed Poisson conformal algebra, Hom-Lie conformal algebra, Poisson conformal algebra, Lie conformal algebra}

\begin{abstract}
The aim of this paper is to introduce the notion of (noncommutative) transposed Poisson conformal algebras, which serve as the conformal analogues of transposed Poisson algebras
and admit a rich class of identities.  We show that the tensor product of two transposed Poisson conformal algebras is also a transposed Poisson conformal algebra.
Moreover, we establish a close relationship between transposed Poisson conformal algebras and Hom-Lie conformal algebras, and give the compatibility conditions between a Poisson conformal algebra and a transposed Poisson conformal algebra.
In addition, we provide several constructions of transposed Poisson conformal algebras arising from related algebraic structures.
Finally, a complete classification of compatible noncommutative transposed Poisson conformal algebraic structures over the Lie conformal algebra $W(a,b)$ is given.
\end{abstract}

\maketitle



\section{Introduction}
Poisson algebras originally emerged from the study of Hamiltonian mechanics \cite{Arn78,Mac70} and Poisson geometry \cite{Vais20,Wein77} and have since appeared in various areas of mathematics and mathematical physics, including
Poisson manifolds \cite{Bha88,Vai94}, algebraic geometry \cite{Bell17,Pol97}, quantum groups \cite{Cha94,Etin98}, quantization theory \cite{Hue90,Kont03}, as well as classical and quantum mechanics \cite{Bay77,Dir67}.
A dual notion of the Poisson algebra, called the transposed Poisson algebra, was introduced by Bai et al.~\cite{Bai23} by reversing the roles of the two bilinear operations in the Leibniz rule defining the Poisson algebra.
A transposed Poisson algebra defined in this way not only parallels the classical Poisson algebras at the algebraic, operadic and categorical levels, but also relates closely to several important algebraic structures such as Novikov-Poisson algebras, pre-Lie Poisson algebras and pre-Lie commutative algebras.

The notion of a Lie conformal algebra, introduced by Kac \cite{Kac98}, encodes an axiomatic description of the singular part of the operator product expansion of chiral fields in two-dimensional conformal field theory.
Closely related to vertex algebras and formal distribution Lie algebras \cite{Chen25,Li96}, Lie conformal algebras have crucial applications in other areas of algebra and integrable systems \cite{Zelm20,Zelm03}. In particular, they provide useful tools in the study of associative algebras, infinite-dimensional Lie algebras satisfying the sole locality property \cite{Su13} and Hamiltonian formalism in the theory of nonlinear evolution equations \cite{Bara09}.
The structure theory, representation theory and cohomology theory of finite Lie conformal algebras have been extensively investigated in \cite{BKV99,CK97,DAK98,SXY20,WY17}.
Significantly, associative conformal algebras naturally emerged from the study of the representation theory of Lie conformal algebras.

The notion of a Poisson conformal algebra was introduced by Kolesnikov \cite{Kol20}, which consists of a Lie conformal algebra and a commutative associative conformal algebra satisfying certain compatibility conditions.
It was shown in \cite{Kol20,LZ23} that Poisson conformal algebras are closely connected to representations of Lie conformal algebras as well as to Poisson-Gel'fand-Dorfman algebras.
The concept of noncommutative Poisson conformal algebras was also studied in \cite{LZ23}, where their cohomology and deformation quantization were investigated.
As a consequence, it was proved that Poisson conformal algebras are the semi-classical limits of conformal formal deformations of commutative associative conformal algebras.
Furthermore, the notion of Poisson conformal bialgebras, characterized by Manin triples of Poisson conformal algebras, was introduced and investigated in \cite{HB24}.

The study of conformal analogues of classical algebraic structures has attracted increasing attention in recent years \cite{sym24,HB2021}, extending classical results to the conformal setting and promoting the development of conformal algebra theory.
As previously mentioned, Poisson conformal algebras and Poisson conformal bialgebras can be regarded as the conformal analogues of Poisson algebras and Poisson bialgebras, respectively.
Moreover, a conformal analogue of $L$-dendriform algebras was presented in \cite{HY25} with the notion of $L$-dendriform conformal algebras, which emerge naturally from the study of $\mathcal{O}$-operators on left-symmetric conformal algebras and solutions to the conformal $S$-equation.
Motivated by these results, it is natural to consider a conformal analogue of transposed Poisson algebras, namely transposed Poisson conformal algebras, which consist of a Lie conformal algebra and a commutative associative conformal algebra satisfying a transposed conformal Leibniz rule.

The rest of this paper is organized as follows. In Section~\ref{sec2}, we review some related notions and basic results about Lie conformal algebras and Poisson conformal algebras.
In Section~\ref{sec3}, we first introduce the notion of (noncommutative) transposed Poisson conformal algebra, which admits a rich class of identities.
Then we show that the tensor product of two transposed Poisson conformal algebras is also a transposed Poisson conformal algebra.
Furthermore, we display a close relationship between transposed Poisson conformal algebras and Hom-Lie conformal algebras, and give the compatibility conditions between a Poisson conformal algebra and a transposed Poisson conformal algebra.
In Section~\ref{sec4}, we provide several constructions of transposed Poisson conformal algebras arising from related algebraic structures.
In Section~\ref{sec5}, we present a complete classification of compatible noncommutative transposed Poisson conformal algebraic structures over the Lie conformal algebra $W(a,b)$.

Throughout this paper, we denote by $\mathbb{C}$ the set of complex numbers, $\mathbb{Z}_{+}$ the set of nonnegative integers. All vector spaces, linear maps, and tensor products are considered over the complex field $\mathbb{C}$.
For any vector space $A$, the space of polynomials of $\lambda$ with coefficients in $A$ is denoted by $A[\lambda]$.

\section{Preliminaries}\label{sec2}
In this section, we recall some basic notions and results concerning Lie conformal algebras and Poisson conformal algebras.
\begin{definition}[\cite{Kac99}]
A {\bf conformal algebra} $(A,\circ_\lambda)$ is a $\mathbb{C}[\partial]$-module $A$ endowed with a $\mathbb{C}$-bilinear map $A\otimes A\to A[\lambda]$, denoted by $a\otimes b\mapsto a \circ_{\lambda}b$, satisfying the following condition:
   \begin{align}
        (\partial a)\circ_{\lambda}b=-\lambda a\circ_{\lambda}b,~~~a\circ_{\lambda}(\partial b)=(\partial+\lambda)a\circ_{\lambda}b.
   \end{align}
\end{definition}
A \emph{derivation} of a conformal algebra $(A,\circ_\lambda)$ is a linear map $D: A\rightarrow A$ such that $D\partial=\partial D$ and $D(a\circ_{\lambda} b)=D(a)\circ_{\lambda} b+a\circ_{\lambda} D(b)$ for all $a,b~\in A$.
For instance, $D=\partial$ is a derivation of every conformal algebra.
\begin{definition}[\cite{Kac98}]
An {\bf associative conformal algebra} $(A,\circ_\lambda)$ is a $\mathbb{C}[\partial]$-module $A$ endowed with a $\mathbb{C}$-bilinear map $A\otimes A\to A[\lambda]$, denoted by $a\otimes b\mapsto a \circ_{\lambda}b$ such that $(A,\circ_\lambda)$ is a conformal algebra and satisfies the following identity:
   \begin{align}\label{ly}
      a\circ_{\lambda}\left(b\circ_{\mu} c\right)=\left(a\circ_{\lambda} b\right)\circ_{\lambda+\mu} c,\quad \forall ~a,b,c\in {A}.
   \end{align}
An associative conformal algebra $(A,\circ_\lambda)$ is called {\bf commutative}, if it satisfies
   \begin{align}
        a\circ_\lambda b=b\circ_{-\partial-\lambda}a,\quad \forall ~a,b\in {A}.
   \end{align}
\end{definition}

\begin{remark}[{\cite{DAK98}}]\label{rem1}
   The following properties always hold in an associative conformal algebra $(A,\circ_\lambda)$:
   \begin{align}
         a\circ_{\lambda}(b\circ_{-\partial-\mu}c) & =(a\circ_{\lambda}b)\circ_{-\partial-\mu}c, \\
         a\circ_{-\partial-\lambda}(b\circ_{\mu}c) & =(a\circ_{-\partial-\mu}b)\circ_{-\partial+\mu-\lambda}c, \\
         a\circ_{-\partial-\lambda}(b\circ_{-\partial-\mu}c) & =(a_{-\partial+\mu-\lambda}b)\circ_{-\partial-\mu}c.
   \end{align}
\end{remark}
\begin{remark}[{\cite{KN23}}] \label{rem3}
   In every commutative associative conformal algebra $(A,\circ_\lambda)$, the following identity holds:
   \begin{align}
      a\circ_\lambda\left(b\circ_\mu c\right)=b\circ_\mu\left(a\circ_\lambda c\right),\quad \forall ~a,b,c\in {A}.
   \end{align}
\end{remark}

\begin{definition}[{\cite{BKV99}}]
A {\bf Lie conformal algebra} $(A,[\cdot_{\lambda}\cdot])$ is a $\mathbb{C}[\partial]$-module $A$ endowed with a $\mathbb{C}$-bilinear map $A\otimes A\to A[\lambda]$, denoted by $a\otimes b\mapsto [a_{\lambda}b]$, satisfying the following conditions:
   \begin{align}
        &[(\partial a)_{\lambda}b]=-\lambda[a_{\lambda}b],~[a_{\lambda}(\partial b)]=(\partial+\lambda)[a_{\lambda}b], &&\text{(Conformal sesquilinearity)}\\
        &[a_{\lambda}b]=-[b_{-\lambda-\partial}a], &&\text{(Skew-symmetry)}\\
        &[a_{\lambda}[b_{\mu}c]]=[[a_{\lambda}b]_{\lambda+\mu}c]+[b_{\mu}[a_{\lambda}c]], \quad \forall ~a,b,c\in {A}.&&\text{(Jacobi identity)} \label{Jacobi}
   \end{align}
A Lie conformal algebra is called {\bf finite} if it is finitely generated as a $\mathbb{C}[\partial]$-module.
\end{definition}

\begin{example}[{\cite{DAK98}}]
The {\bf Virasoro Lie conformal algebra} $Vir$ is the simplest nontrivial Lie conformal algebra. It is defined by
   \begin{align}
        Vir=\mathbb{C}[\partial]L,\quad[L_ \lambda L]=(\partial+2\lambda)L.
   \end{align}
\end{example}
\begin{example}[{\cite{LHW20}}]
   The {\bf Lie conformal algebra} $W(a,b)$ with two parameters $a,b\in \mathbb{C}$ is a free $\mathbb{C}[\partial]$-module generated by $L$ and $M$ satisfying
   \begin{align*}
        [L_ \lambda L]=(\partial+2\lambda)L,\quad[L_ \lambda M]=(\partial+a\lambda+b)M,\quad[M_ \lambda M]=0.
   \end{align*}
\end{example}

\begin{definition}[{\cite{Yuan14}}]
A {\bf Hom-Lie conformal algebra} $(A,[\cdot_{\lambda}\cdot],\alpha)$ is a $\mathbb{C}[\partial]$-module $A$ endowed with a linear endomorphism $\alpha$ such that $\alpha\partial=\partial\alpha$, and a $\lambda$-bracket $[\cdot_{\lambda}\cdot]$ which is a $\mathbb{C}$-bilinear map $A \otimes A \rightarrow A [\lambda]$ such that the following axioms hold for all $a,b,c \in A$:
   \begin{align}
      &{\left[(\partial a)_\lambda b\right]=-\lambda\left[a_\lambda b\right],\left[a_\lambda (\partial b)\right] }=(\partial+\lambda)\left[a_\lambda b\right], &&\text{(Conformal sesquilinearity)} \\
      &{\left[a_\lambda b\right] } =-\left[b_{-\lambda-\partial} a\right], &&\text{(Skew-symmetry)}\\
      &{\left[\alpha(a)_\lambda\left[b_\mu c\right]\right] } =[\left[a_\lambda b\right]_{\lambda+\mu} \alpha(c)]+\left[\alpha(b)_\mu\left[a_\lambda c\right]\right], &&\text{(Hom-Jacobi identity)}
   \end{align}
\end{definition}

\begin{definition}[\cite{Kac98}]
A {\bf left-symmetric conformal algebra} (also known as a {\bf pre-Lie conformal algebra}) $(A,\ast_\lambda)$ is a $\mathbb{C}[\partial]$-module $A$ endowed with a $\mathbb{C}$-bilinear map $A\otimes A\to A[\lambda]$, denoted by $a\otimes b\mapsto a \ast_{\lambda}b$ such that $(A,\ast_\lambda)$ is a conformal algebra and satisfies 
   \begin{align}
        (a\ast_\lambda b)\ast_{\lambda+\mu}c-a\ast_\lambda(b\ast_\mu c)=(b\ast_\mu a)\ast_{\lambda+\mu}c-b\ast_\mu(a\ast_\lambda c),\quad \forall ~a,b,c\in {A}.
   \end{align}
A {\bf Novikov conformal algebra} $(A,\ast_\lambda)$ is a left-symmetric conformal algebra satisfying
   \begin{align}
      (a\ast_\lambda b)\ast_{\lambda+\mu}c=(a\ast_\lambda c)\ast_{-\mu-\partial}b,\quad \forall ~a,b,c\in {A}.
   \end{align}
\end{definition}

\begin{proposition}[{\cite{HL15}}] \label{propo1}
   Associative conformal algebras are left-symmetric conformal algebras. If $(A,\ast_\lambda)$ is a left-symmetric conformal algebra, then the $\lambda$-bracket
   \begin{align*}
      [a_\lambda b]=a\ast_\lambda b-b\ast_{-\lambda-\partial}a,\quad \forall~ a,b \in {A}
   \end{align*}
   defines a Lie conformal algebra $\mathfrak{g}(A)$, which is called a \textbf{sub-adjacent Lie conformal algebra} of $A$.
\end{proposition}

\begin{definition}[{\cite{LZ23}}]
A \textbf{Poisson conformal algebra} (PCA) is a $\mathbb{C}[\partial]$-module $A$ endowed with two $\mathbb{C}$-bilinear maps $A \otimes A \rightarrow A[\lambda]$, denoted by $a\otimes b\mapsto a \circ_{\lambda}b$ and $a\otimes b\mapsto [a_{\lambda}b]$, respectively, such that $(A,\circ_\lambda)$ is a {commutative associative conformal algebra}, $(A,[{\cdot_ \lambda \cdot}])$ is a {Lie conformal algebra} and the following \emph{conformal Leibniz rule} holds:
   \begin{align}\label{rule1}
      [a_{\lambda}(b\circ_{\mu}c)]=[a_{\lambda}b]\circ_{\lambda+\mu}c+b\circ_{\mu}[a_{\lambda}c],\quad\forall~ a,b,c\in A.		
   \end{align}
If the associative $\lambda$-multiplication in the Poisson conformal algebra $(A,\circ_{\lambda},[{\cdot_ \lambda \cdot}])$ is noncommutative, we call it a {\textbf{noncommutative Poisson conformal algebra}}.
\end{definition}

\begin{remark}[{\cite{Kol20}}]
Relation \eqref{rule1} is equivalent to
   \begin{equation}\label{rule6}
      [(a \circ_\lambda b)_\mu c] = b \circ_{\mu - \lambda} [a _\lambda c] + a \circ_\lambda [b _{\mu - \lambda} c],\quad\forall~ a,b,c\in A.
   \end{equation}
\end{remark}

\begin{definition}[{\cite{Bai23}}]
Let $L$ be a vector space equipped with two bilinear operations
$$\cdot,\,[~,~ ] : L \otimes  L \rightarrow L.$$
The triple $ (L, \cdot, [~ , ~])$ is called a \textbf{transposed Poisson algebra}, if $ (L, \cdot)$ is a commutative associative algebra and $(L, [~,~])$ is a Lie algebra that satisfy the following compatibility condition
\begin{align}\label{tpa}
	2z \cdot [x, y] = [z \cdot x, y] + [x, z \cdot y], \quad\forall ~ x, y, z \in L.
\end{align}
\end{definition}
Eq.~\eqref{tpa} is called the transposed Leibniz rule because the roles played by the two binary operations in the Leibniz rule in a Poisson algebra are switched. Further, the resulting operation is rescaled by introducing a factor 2 on the left hand side.

\section{Transposed Poisson conformal algebras}\label{sec3}

In this section, we introduce the notion of a transposed Poisson conformal algebra, which is a conformal analogue of the transposed Poisson algebra. Then we present some identities and properties of transposed Poisson conformal algebras.

\subsection{Definition of transposed Poisson conformal algebras}

\begin{definition}
   A \textbf{transposed Poisson conformal algebra} (TPCA)  is a $\mathbb{C}[\partial]$-module $\mathcal{L}$ endowed with two $\mathbb{C}$-bilinear maps
   $$\circ_{\lambda},\, [{\cdot_ \lambda \cdot}]: \mathcal{L}  \otimes \mathcal{L}  \rightarrow \mathcal{L} [\lambda],$$ such that $(\mathcal{L} ,\circ_\lambda)$ is a {commutative associative conformal algebra}, $(\mathcal{L} ,[{\cdot_ \lambda \cdot}])$ is a {Lie conformal algebra} and the following \emph{transposed conformal Leibniz rule} holds:
   \begin{align}\label{rule3}
      2 (a\circ_{\lambda}[b_{\mu}c])=[(a\circ_{\lambda}b)_{\lambda+\mu}c]+[b_{\mu}(a\circ_{\lambda}c)],\quad\forall~ a,b,c\in \mathcal{L}.	
   \end{align}
   If the associative $\lambda$-multiplication in the transposed Poisson conformal algebra $(\mathcal{L},\circ_{\lambda},[{\cdot_ \lambda \cdot}])$ is noncommutative, we call it a \textbf{noncommutative TPCA}.  
\end{definition}

\begin{remark}
	The relation \eqref{rule3} can be equivalently written as
	\begin{align}\label{rule5}
		2([a_{\lambda}b]\circ_{\mu}c)=[a_{{\lambda}}(b\circ_{\mu-\lambda}c)]-[b_{\mu-\lambda}(a\circ_{\lambda}c)],\quad\forall~ a,b,c\in \mathcal{L}.
	\end{align}
\end{remark}

Let $(\mathcal{L},\circ_{\lambda},[{\cdot_ \lambda \cdot}])$ be a (noncommutative) transposed Poisson conformal algebra.
Define the $n$-th products on the (commutative) associative conformal algebra $(\mathcal{L},\circ_{\lambda})$ and the Lie conformal algebra $(\mathcal{L},[{\cdot_ \lambda \cdot}])$ for $n\in \mathbb{Z}_{+}$, respectively, by
   \begin{equation*}
      a\circ_\lambda b=\sum_{n\in \mathbb{Z}_{+}}\lambda^{(n)}(a_{(n)}b),\quad[a_\lambda b]=\sum_{n\in \mathbb{Z}_{+}}\lambda^{(n)}(a_{[n]}b),
   \end{equation*}
where $\lambda^{(n)}=\lambda^{n}/{n!}$ and $a,b\in \mathcal{L}$. Then, the notions of (commutative) associative conformal algebra and Lie conformal algebra can be equivalently expressed in terms of their respective $n$-th products.
Moreover, the transposed conformal Leibniz rule of (noncommutative) transposed Poisson conformal algebras can be rewritten in terms of $n$-th products of associative conformal algebras and Lie conformal algebras as follows:
   \begin{equation}\label{rule4}
      2(a_{(n)}(b_{[m]}c))=\sum_{j=0}^n\binom{n}{j}(a_{(j)}b)_{[n+m-j]}c+b_{{[m]}}(a_{(n)}c),\quad\forall~ a,b,c\in \mathcal{L}.
   \end{equation}

\begin{example}
  Let $(A,\circ,[\cdot,\cdot])$ be an ordinary transposed Poisson algebra. Then $\mathcal{L}=\mathbb{C}[\partial]\otimes A$ equipped with
  $$a \circ_{\lambda} b=a \circ b,\quad[{a_ \lambda b}]=[a,b],\quad\forall~ a,b\in A$$
  is a transposed Poisson conformal algebra.
\end{example}

\begin{example}
Let $(\mathcal{L}_1 ,\circ_{1\lambda},[\cdot_{\lambda}\cdot]_1,\partial_1)$ and $(\mathcal{L}_2 ,\circ_{2\lambda},[\cdot_{\lambda}\cdot]_2,\partial_2)$ be two TPCAs. For all $x_1,y_1 \in \mathcal{L}_1$, $x_2,y_2 \in \mathcal{L}_2,$ define
\begin{align*}
(x_1+x_2)\circ_{\lambda}(y_1+y_2)&=x_1\circ_{1\lambda}y_1+x_2\circ_{2\lambda}y_2, \\
[(x_1+x_2)_{\lambda}(y_1+y_2)]&=[{x_1}_{\lambda}y_1]_1+[{x_2}_{\lambda}y_2]_2.
\end{align*}
Then $(\mathcal{L}_1 \oplus \mathcal{L}_2, \circ_{\lambda}, [\cdot_{\lambda}\cdot], \partial=\partial_1\oplus\partial_2)$ is a TPCA.
\end{example}

\subsection{Identities in transposed Poisson conformal algebras}
There is a rich class of identities for transposed Poisson conformal algebras. 
\begin{theorem}
   Let $(\mathcal{L} ,\circ_\lambda,[{\cdot_ \lambda \cdot}])$ be a transposed Poisson conformal algebra. Then the following identities hold:
   \begin{align}
      &x\circ_{\lambda}[y_{\mu}z]+y\circ_{\mu}[z_{-\partial-\lambda}x]+z\circ_{-\partial-\lambda-\mu}[x_{\lambda}y]=0, \label{eqh}\\
      &[x_{\lambda}y]\circ_{\mu}z+[y_{\mu-\lambda}z]\circ_{-\partial-\lambda}x+[z_{-\partial-\lambda}x]\circ_{-\partial-\mu+\lambda}y=0, \label{eqw}\\
      &[(h\circ_{\lambda}[x_{\gamma}y])_{\lambda+\mu}z]+[(h\circ_{\lambda}[y_{\mu-\gamma}z])_{-\partial-\gamma}x]+[(h\circ_{\lambda}[z_{-\partial-\gamma}x])_{-\partial-\mu+\gamma}y]=0, \label{eqs}\\
      &[[x_{\gamma}y]_{\mu}(h\circ_{\lambda}z)]+[[y_{\mu-\gamma}z]_{-\partial-\lambda-\gamma}(h\circ_{\lambda}x)]+[[z_{-\partial-\gamma}x]_{-\partial-\lambda-\mu+\gamma}(h\circ_{\lambda}y)]=0, \label{equ}\\
      &[h_{\lambda}x]\circ_{\lambda+\gamma}[y_{\mu-\gamma}z]+[h_{\lambda}y]\circ_{\lambda+\mu-\gamma}[z_{-\partial-\gamma}x]+[h_{\lambda}z]\circ_{-\partial-\mu}[x_{\gamma}y]=0,  \label{eql}\\
      &[(u\circ_{\lambda}x)_{\mu}(v\circ_{\gamma}y)]+[(v\circ_{\gamma}x)_{\gamma+\mu-\lambda}(u\circ_{\lambda}y)]=2\left(u\circ_{\lambda}v\right)\circ_{\lambda+\gamma}[x_{\mu-\lambda}y], \label{ss1}\\
      &x\circ_{\mu-\lambda}[u_{\lambda}(v\circ_{\gamma}y)]+[(v\circ_{\gamma}x)_{\gamma+\mu-\lambda}u]\circ_{\gamma+\mu}y
      =\left(u\circ_{\lambda}v\right)\circ_{\lambda+\gamma}[x_{\mu-\lambda}y],\label{ss2}
   \end{align}
   for all $x,y,z,h,u,v \in \mathcal{L}$.
\end{theorem}
\begin{proof}
Let $x,y,z,h,u,v \in \mathcal{L}$. In fact, Eq.~\eqref{eqw} is equivalent to Eq.~\eqref{eqh}. By Eqs.~\eqref{rule3} and \eqref{rule5}, we have
   \begin{align*}
      &{2} (x\circ_{\lambda}[y_{\mu}z])=[(x\circ_{\lambda}y)_{\lambda+\mu}z]+[y_{\mu}(x\circ_{\lambda}z)],	\\
      &{2} (y\circ_{\mu}[z_{-\partial-\lambda}x])=-{2} (y\circ_{\mu}[x_{\lambda}z])=-[(y\circ_{\mu}x)_{\lambda+\mu}z]-[x_{\lambda}(y\circ_{\mu}z)],\\
      &{2} (z\circ_{-\partial-\lambda-\mu}[x_{\lambda}y])={2} ([x_{\lambda}y]\circ_{\lambda+\mu}z)=[x_{{\lambda}}(y\circ_{\mu}z)]-[y_{\mu}(x\circ_{\lambda}z)].
   \end{align*}
Note that $[(x\circ_{\lambda}y)_{\lambda+\mu}z]=[(y\circ_{-\partial-\lambda}x)_{\lambda+\mu}z]=[(y\circ_{\mu}x)_{\lambda+\mu}z]$. Taking the sum of the three identities yields Eq.~\eqref{eqh}.

The proof of Eq.~\eqref{eqs} is more involved. First by Eq.~\eqref{rule3}, we have
   \begin{align*}
      2(h\circ_{\lambda}[[x_{\gamma}y]_{\mu}z])&=[(h\circ_{\lambda}[x_{\gamma}y])_{\lambda+\mu}z]+[[x_{\gamma}y]_{\mu}(h\circ_{\lambda}z)],\\
      -2(h\circ_{\lambda}[x_{\gamma}[y_{\mu-\gamma}z]])&=[(h\circ_{\lambda}[y_{\mu-\gamma}z])_{-\partial-\gamma}x]+[[y_{\mu-\gamma}z]_{-\partial-\lambda-\gamma}(h\circ_{\lambda}x)],\\
      2(h\circ_{\lambda}[y_{\mu-\gamma}[x_{\gamma}z]])&=[(h\circ_{\lambda}[z_{-\partial-\gamma}x])_{-\partial-\mu+\gamma}y]+[[z_{-\partial-\gamma}x]_{-\partial-\lambda-\mu+\gamma}(h\circ_{\lambda}y)].
   \end{align*}
Upon summing the three identities above and applying the Jacobi identity~\eqref{Jacobi}, we obtain
   \begin{eqnarray}\label{Id1}
      \begin{aligned}
      &[(h\circ_{\lambda}[x_{\gamma}y])_{\lambda+\mu}z]+[(h\circ_{\lambda}[y_{\mu-\gamma}z])_{-\partial-\gamma}x]+[(h\circ_{\lambda}[z_{-\partial-\gamma}x])_{-\partial-\mu+\gamma}y]+ \\
      &[[x_{\gamma}y]_{\mu}(h\circ_{\lambda}z)]+[[y_{\mu-\gamma}z]_{-\partial-\lambda-\gamma}(h\circ_{\lambda}x)]+[[z_{-\partial-\gamma}x]_{-\partial-\lambda-\mu+\gamma}(h\circ_{\lambda}y)]=0.
      \end{aligned}
   \end{eqnarray}
Then, applying the Jacobi identity to $[[x_{\gamma}y]_{\mu}(h\circ_{\lambda}z)]$ yields
   \begin{align*}
      [[x_{\gamma}y]_{\mu}(h\circ_{\lambda}z)]+[y_{\mu-\gamma}[x_{\gamma}(h\circ_{\lambda}z)]]=[x_{\gamma}[y_{\mu-\gamma}(h\circ_{\lambda}z)]]
   \end{align*}
and applying Eq.~\eqref{rule3} to $[y_{\mu-\gamma}(h\circ_{\lambda}z)]$ yields
   \begin{align*}
      [y_{\mu-\gamma}(h\circ_{\lambda}z)]=2\left(h\circ_{\lambda}[y_{\mu-\gamma}z]\right)-[(h\circ_{\lambda}y)_{\lambda+\mu-\gamma}z].
   \end{align*}
Hence, we arrive at
   \begin{align*}
      [[x_{\gamma}y]_{\mu}(h\circ_{\lambda}z)]+2[(h\circ_{\lambda}[y_{\mu-\gamma}z])_{-\partial-\gamma}x]=-[x_{\gamma}[(h\circ_{\lambda}y)_{\lambda+\mu-\gamma}z]]-[y_{\mu-\gamma}[x_{\gamma}(h\circ_{\lambda}z)]].
   \end{align*}
In a similar way, we have
   \begin{align*}
      [[y_{\mu-\gamma}z]_{-\partial-\lambda-\gamma}(h\circ_{\lambda}x)]+2[(h\circ_{\lambda}&[z_{-\partial-\gamma}x])_{-\partial-\mu+\gamma}y] \nonumber\\
      &=[y_{\mu-\gamma}[x_{\gamma}(h\circ_{\lambda}z)]]-[[(h\circ_{\lambda}x)_{\lambda+\gamma}y]_{\lambda+\mu}z],\\
      [[z_{-\partial-\gamma}x]_{-\partial-\lambda-\mu+\gamma}(h\circ_{\lambda}y)]+2[(h\circ_{\lambda}&[x_{\gamma}y])_{\lambda+\mu}z] \nonumber\\
      &=[x_{\gamma}[(h\circ_{\lambda}y)_{\lambda+\mu-\gamma}z]]+[[(h\circ_{\lambda}x)_{\lambda+\gamma}y]_{\lambda+\mu}z].
   \end{align*}
Upon summing the three identities above, we obtain
   \begin{eqnarray}
      \begin{aligned}\label{Id2}
      &2\left([(h\circ_{\lambda}[x_{\gamma}y])_{\lambda+\mu}z]+[(h\circ_{\lambda}[y_{\mu-\gamma}z])_{-\partial-\gamma}x]+[(h\circ_{\lambda}[z_{-\partial-\gamma}x])_{-\partial-\mu+\gamma}y]\right)+ \\
      &[[x_{\gamma}y]_{\mu}(h\circ_{\lambda}z)]+[[y_{\mu-\gamma}z]_{-\partial-\lambda-\gamma}(h\circ_{\lambda}x)]+[[z_{-\partial-\gamma}x]_{-\partial-\lambda-\mu+\gamma}(h\circ_{\lambda}y)]=0.
      \end{aligned}
   \end{eqnarray}
Taking the difference between Eq.~\eqref{Id1} and Eq.~\eqref{Id2} yields Eq.~\eqref{eqs}.
Moreover, Eq.~\eqref{equ} follows by substituting Eq.~\eqref{eqs} into Eq.~\eqref{Id1}.

By Eq.~\eqref{rule3}, we obtain
   \begin{align}\label{yqq2}
      2\left([x_{\gamma}y]\circ_{\mu}[h_{\lambda}z]\right)=[([x_{\gamma}y]\circ_{\mu}h)_{\lambda+\mu}z]+[h_{\lambda}([x_{\gamma}y]\circ_{\mu}z)].
   \end{align}
Observe that
   \begin{align*}
      [(h\circ_{\lambda}[x_{\gamma}y])_{\lambda+\mu}z]=[([x_{\gamma}y]\circ_{-\partial-\lambda}h)_{\lambda+\mu}z]
      =[([x_{\gamma}y]\circ_{\mu}h)_{\lambda+\mu}z].
   \end{align*}
Thus, Eq.~\eqref{yqq2} can be rewritten as
   \begin{align*}
      2\left([h_{\lambda}z]\circ_{-\partial-\mu}[x_{\gamma}y]\right)=[(h\circ_{\lambda}[x_{\gamma}y])_{\lambda+\mu}z]+[h_{\lambda}(z\circ_{-\partial-\mu}[x_{\gamma}y])].
   \end{align*}
Similarly, by Eq.~\eqref{rule5}, we have
   \begin{align*}
      2([h_{\lambda}x]\circ_{\lambda+\gamma}[y_{\mu-\gamma}z])&=[(h\circ_{\lambda}[y_{\mu-\gamma}z])_{-\partial-\gamma}x]+[h_{{\lambda}}(x\circ_{\gamma}[y_{\mu-\gamma}z])],\\
      2([h_{\lambda}y]\circ_{\lambda+\mu-\gamma}[z_{-\partial-\gamma}x])&=[(h\circ_{\lambda}[z_{-\partial-\gamma}x])_{-\partial-\mu+\gamma}y]+[h_{{\lambda}}(y\circ_{\mu-\gamma}[z_{-\partial-\gamma}x])].
   \end{align*}
Upon summing the three identities above and applying Eqs.~\eqref{eqh} and \eqref{eqs}, we get Eq.~\eqref{eql}.

Finally, we turn to the proof of Eqs.~\eqref{ss1} and \eqref{ss2}. Since $(\mathcal{L} ,\circ_\lambda)$ is a {commutative associative conformal algebra}, we have
   \begin{align*}
        \left(u\circ_{\lambda}v\right)\circ_{\lambda+\gamma}[x_{\mu-\lambda}y]=u\circ_{\lambda}\left(v\circ_{\gamma}[x_{\mu-\lambda}y]\right).
   \end{align*}
By Eq.~\eqref{rule3}, we obtain
   \begin{align}\label{qq1}
      {4}(u\circ_{\lambda}\left(v\circ_{\gamma}[x_{\mu-\lambda}y]\right))=[(u&\circ_{\lambda}x)_{\mu}(v\circ_{\gamma}y)]+[(v\circ_{\gamma}x)_{\gamma+\mu-\lambda}(u\circ_{\lambda}y)] \nonumber\\
      &+[(u\circ_{\lambda}(v\circ_{\gamma}x))_{\gamma+\mu}y]+[x_{\mu-\lambda}(u\circ_{\lambda}(v\circ_{\gamma}y))].
   \end{align}
Observe that
   \begin{align}\label{qq2}
      {2}\left(u\circ_{\lambda}v\right)\circ_{\lambda+\gamma}[x_{\mu-\lambda}y]&=[((u\circ_{\lambda}v)\circ_{\lambda+\gamma}x)_{\mu+\gamma}y]+[x_{\mu-\lambda}((u\circ_{\lambda}v)\circ_{\lambda+\gamma}y)] \nonumber\\
      &=[(u\circ_{\lambda}(v\circ_{\gamma}x))_{\gamma+\mu}y]+[x_{\mu-\lambda}(u\circ_{\lambda}(v\circ_{\gamma}y))].
   \end{align}
Then Eq.~\eqref{ss1} follows from taking the difference between Eq.~\eqref{qq1} and Eq.~\eqref{qq2}.

Note that the first term on the left-hand side of Eq.~\eqref{ss1} can be rewritten as
   \begin{align*}
      [(u\circ_{\lambda}x)_{\mu}(v\circ_{\gamma}y)]=[(x\circ_{-\partial-\lambda}u)_{\mu}(v\circ_{\gamma}y)]
      =[(x\circ_{\mu-\lambda}u)_{\mu}(v\circ_{\gamma}y)]
   \end{align*}
and by Eq.~\eqref{rule3}, we obtain
   \begin{align}\label{f5}
      [(x\circ_{\mu-\lambda}u)_{\mu}(v\circ_{\gamma}y)]={2}(x\circ_{\mu-\lambda}[u_{\lambda}(v\circ_{\gamma}y)])-[u_{\lambda}(x\circ_{\mu-\lambda}(v\circ_{\gamma}y))].
   \end{align}
For the second term, by Eq.~\eqref{rule5}, we get
   \begin{align}\label{f6}
      [(v\circ_{\gamma}x)_{{\gamma+\mu-\lambda}}(u\circ_{\lambda}y)]={2}([(v\circ_{\gamma}x)_{\gamma+\mu-\lambda}u]\circ_{\gamma+\mu}y)+[u_{\lambda}((v\circ_{\gamma}x)\circ_{\gamma+\mu-\lambda}y)].
   \end{align}
Since
   \begin{align*}
      x\circ_{\mu-\lambda}(v\circ_{\gamma}y)&=(x\circ_{\mu-\lambda}v)\circ_{\gamma+\mu-\lambda}y=(v\circ_{-\partial-\mu+\lambda}x)\circ_{\gamma+\mu-\lambda}y
      =(v\circ_{\gamma}x)\circ_{\gamma+\mu-\lambda}y,
   \end{align*}
we have 
   \begin{align*}
        [u_{\lambda}(x\circ_{\mu-\lambda}(v\circ_{\gamma}y))]=[u_{\lambda}((v\circ_{\gamma}x)\circ_{\gamma+\mu-\lambda}y)].
   \end{align*}
Hence, by taking the sum of Eqs.~\eqref{f5} and \eqref{f6} and applying Eq.~\eqref{ss1}, we arrive at Eq.~\eqref{ss2}.
\end{proof}

\subsection{Tensor products of transposed Poisson conformal algebras}
Now we define the tensor product of transposed Poisson conformal algebras as follows:
\begin{theorem}
Let $(\mathcal{L}_1 ,\circ^{1}_{\lambda},[\cdot_{\lambda}\cdot]^1)$ and $(\mathcal{L}_2 ,\circ^{2}_{\lambda},[\cdot_{\lambda}\cdot]^2)$ be two transposed Poisson conformal algebras.
Set $\mathcal L:=\mathcal L_1\otimes_{\mathbb C[\partial]} \mathcal L_2$ and the $\mathbb C[\partial]$-module structure on $\mathcal L$ is given by
\begin{align*}
\partial(a\otimes b):=(\partial a)\otimes b = a\otimes (\partial b),
\quad a\in\mathcal L_1,\, b\in\mathcal L_2 .
\end{align*}
Define the $\lambda$-product and the $\lambda$-bracket on $\mathcal L$ respectively by
\begin{gather*}
(a_1\otimes a_2)\circ_\lambda (b_1\otimes b_2)
= (a_1\circ^{1}_\lambda b_1)\otimes (a_2\circ^{2}_\lambda b_2),\\
[(a_1\otimes a_2)_\lambda (b_1\otimes b_2)]
= [{a_1}_{\lambda}b_1]^{1}\otimes (a_2\circ^{2}_\lambda b_2)+(a_1\circ^{1}_\lambda b_1)\otimes [{a_2}_{\lambda}b_2]^{2},
\end{gather*}
for all $a_1,b_1 \in \mathcal{L}_1$, $a_2,b_2 \in \mathcal{L}_2$.
Then $(\mathcal L,\circ_\lambda,[\cdot_\lambda\cdot])$ is a transposed Poisson conformal algebra.
\end{theorem}
\begin{proof}
Let $a_1,b_1,c_1 \in \mathcal{L}_1$, $a_2,b_2,c_2 \in \mathcal{L}_2$. It is easy to verify that $(\mathcal L,\circ_\lambda)$ is a commutative associative conformal algebra.
Then we show that $(\mathcal L,[\cdot_\lambda\cdot])$ is a Lie conformal algebra. 
\begin{itemize}
\item[(a)] {\bf{Conformal sesquilinearity}:}
\begin{align*}
&[\partial(a_1\otimes a_2)_\lambda (b_1\otimes b_2)]=[((\partial a_1)\otimes a_2)_\lambda (b_1\otimes b_2)]
= [{(\partial a_1)}_{\lambda}b_1]^{1}\otimes (a_2\circ^{2}_\lambda b_2)\\
&+((\partial a_1)\circ^{1}_\lambda b_1)\otimes [{a_2}_{\lambda}b_2]^{2}
= -\lambda([{a_1}_{\lambda}b_1]^{1}\otimes (a_2\circ^{2}_\lambda b_2))-\lambda((a_1\circ^{1}_\lambda b_1)\otimes [{a_2}_{\lambda}b_2]^{2})\\
&=-\lambda ([{a_1}_{\lambda}b_1]^{1}\otimes (a_2\circ^{2}_\lambda b_2)+(a_1\circ^{1}_\lambda b_1)\otimes [{a_2}_{\lambda}b_2]^{2})
=-\lambda [(a_1\otimes a_2)_\lambda (b_1\otimes b_2)].
\end{align*}
\item[(b)] {\bf{Skew-symmetry}:} 
\begin{align*}
&[(b_1\otimes b_2)_{-\partial-\lambda} (a_1\otimes a_2)]=[{b_1}_{-\partial-\lambda}a_1]^{1}\otimes (b_2\circ^{2}_{-\partial-\lambda} a_2)+(b_1\circ^{1}_{-\partial-\lambda} a_1)\otimes [{b_2}_{-\partial-\lambda}a_2]^{2}\\
&=-[{a_1}_{\lambda}b_1]^{1}\otimes (a_2\circ^{2}_\lambda b_2)-(a_1\circ^{1}_\lambda b_1)\otimes [{a_2}_{\lambda}b_2]^{2}
=-[(a_1\otimes a_2)_\lambda (b_1\otimes b_2)].
\end{align*}
\item[(c)] {\bf{Jacobi identity}:} For the Jacobi identity, we compute
\begin{align*}
&[[(a_1\otimes a_2)_\lambda (b_1\otimes b_2)]_{\lambda+\mu} (c_1\otimes c_2)]+[(b_1\otimes b_2)_\mu [(a_1\otimes a_2)_\lambda (c_1\otimes c_2)]]\\
&-[(a_1\otimes a_2)_\lambda [(b_1\otimes b_2)_\mu (c_1\otimes c_2)]]
=(Y1)+(Y2)+(Q1)+(Q2),
\end{align*}
\end{itemize}
where
\begin{align*}
(Y1)&=[{[{a_1}_{\lambda}b_1]^{1}}_{\lambda+\mu}c_1]^{1}\otimes ((a_2\circ^{2}_\lambda b_2)\circ^{2}_{\lambda+\mu} c_2)+[{b_1}_{\mu}[{a_1}_{\lambda}c_1]^{1}]^{1}\otimes (b_2\circ^{2}_\mu (a_2\circ^{2}_\lambda c_2))\\
&-[{a_1}_{\lambda}[{b_1}_{\mu}c_1]^{1}]^{1}\otimes (a_2\circ^{2}_\lambda (b_2\circ^{2}_\mu c_2)),\\
(Y2)&=((a_1\circ^{1}_\lambda b_1)\circ^{1}_{\lambda+\mu} c_1)\otimes [{[{a_2}_{\lambda}b_2]^{2}}_{\lambda+\mu}c_2]^{2}+(b_1\circ^{1}_\mu (a_1\circ^{1}_\lambda c_1))\otimes [{b_2}_{\mu}[{a_2}_{\lambda}c_2]^{2}]^{2}\\
&-(a_1\circ^{1}_\lambda (b_1\circ^{1}_\mu c_1))\otimes [{a_2}_{\lambda}[{b_2}_{\mu}c_2]^{2}]^{2},\\
(Q1)&=([{a_1}_{\lambda}b_1]^{1}\circ^{1}_{\lambda+\mu} c_1)\otimes [{(a_2\circ^{2}_\lambda b_2)}_{\lambda+\mu}c_2]^{2}+(b_1\circ^{1}_\mu [{a_1}_{\lambda}c_1]^{1})\otimes [{b_2}_{\mu}(a_2\circ^{2}_\lambda c_2)]^{2}\\
&-(a_1\circ^{1}_\lambda [{b_1}_{\mu}c_1]^{1})\otimes [{a_2}_{\lambda}(b_2\circ^{2}_\mu c_2)]^{2},\\
(Q2)&=[{(a_1\circ^{1}_\lambda b_1)}_{\lambda+\mu}c_1]^{1}\otimes ([{a_2}_{\lambda}b_2]^{2}\circ^{2}_{\lambda+\mu} c_2)+[{b_1}_{\mu}(a_1\circ^{1}_\lambda c_1)]^{1}\otimes (b_2\circ^{2}_\mu [{a_2}_{\lambda}c_2]^{2})\\
&-[{a_1}_{\lambda}(b_1\circ^{1}_\mu c_1)]^{1}\otimes (a_2\circ^{2}_\lambda [{b_2}_{\mu}c_2]^{2}).
\end{align*}

Since $(\mathcal{L}_1 ,\circ^{1}_{\lambda})$, $(\mathcal{L}_2 ,\circ^{2}_{\lambda})$ are commutative associative conformal algebras and 
$(\mathcal{L}_1,[\cdot_{\lambda}\cdot]^1)$, $(\mathcal{L}_2,[\cdot_{\lambda}\cdot]^2)$ are Lie conformal algebras, $(Y1)$ and $(Y2)$ are zero.
By Eq.~\eqref{eqh}, we obtain
\begin{align*}
(b_1\circ^{1}_\mu [{a_1}_{\lambda}c_1]^{1})\otimes [{b_2}_{\mu}(a_2\circ^{2}_\lambda c_2)]^{2}
&=(a_1\circ^{1}_\lambda [{b_1}_{\mu}c_1]^{1}+[{a_1}_{\lambda}b_1]^{1}\circ^{1}_{\lambda+\mu}c_1)\otimes [{b_2}_{\mu}(a_2\circ^{2}_\lambda c_2)]^{2},\\
[{b_1}_{\mu}(a_1\circ^{1}_\lambda c_1)]^{1}\otimes (b_2\circ^{2}_\mu [{a_2}_{\lambda}c_2]^{2})
&=[{b_1}_{\mu}(a_1\circ^{1}_\lambda c_1)]^{1}\otimes (a_2\circ^{2}_\lambda [{b_2}_{\mu}c_2]^{2}+[{a_2}_{\lambda}b_2]^{2}\circ^{2}_{\lambda+\mu}c_2).
\end{align*}
Substituting the above equations into $(Q1)$ and $(Q2)$ respectively, we get
\begin{align*}
(Q1)&=(a_1\circ^{1}_\lambda [{b_1}_{\mu}c_1]^{1})\otimes ([{b_2}_{\mu}(a_2\circ^{2}_\lambda c_2)]^{2}-[{a_2}_{\lambda}(b_2\circ^{2}_\mu c_2)]^{2})\\
&+([{a_1}_{\lambda}b_1]^{1}\circ^{1}_{\lambda+\mu} c_1)\otimes ([{(a_2\circ^{2}_\lambda b_2)}_{\lambda+\mu}c_2]^{2}+[{b_2}_{\mu}(a_2\circ^{2}_\lambda c_2)]^{2}),\\
(Q2)&=([{b_1}_{\mu}(a_1\circ^{1}_\lambda c_1)]^{1}-[{a_1}_{\lambda}(b_1\circ^{1}_\mu c_1)]^{1})\otimes (a_2\circ^{2}_\lambda [{b_2}_{\mu}c_2]^{2})\\
&+([{(a_1\circ^{1}_\lambda b_1)}_{\lambda+\mu}c_1]^{1}+[{b_1}_{\mu}(a_1\circ^{1}_\lambda c_1)]^{1})\otimes ([{a_2}_{\lambda}b_2]^{2}\circ^{2}_{\lambda+\mu} c_2).
\end{align*}
By Eqs.~\eqref{rule3} and \eqref{rule5}, we have
\begin{align*}
(Q1)&=-2(a_1\circ^{1}_\lambda [{b_1}_{\mu}c_1]^{1})\otimes ([{a_2}_{\lambda}b_2]^{2}\circ^2_{\lambda+\mu}c_2)
+2([{a_1}_{\lambda}b_1]^{1}\circ^{1}_{\lambda+\mu} c_1)\otimes (a_2\circ^{2}_\lambda [{b_2}_{\mu}c_2]^{2}),\\
(Q2)&=-2([{a_1}_{\lambda}b_1]^{1}\circ^1_{\lambda+\mu}c_1)\otimes (a_2\circ^{2}_\lambda [{b_2}_{\mu}c_2]^{2})
+2(a_1\circ^{1}_\lambda [{b_1}_{\mu}c_1]^{1})\otimes ([{a_2}_{\lambda}b_2]^{2}\circ^{2}_{\lambda+\mu} c_2).
\end{align*}
Hence $(Q1)+(Q2)=0$. Therefore, $(\mathcal L,[\cdot_\lambda\cdot])$ is a Lie conformal algebra.

Next, we show that the relation \eqref{rule5} holds. We compute separately 
\begin{align*}
&[(a_1\otimes a_2)_\lambda ((b_1\otimes b_2)\circ_\mu (c_1\otimes c_2))]
=[(a_1\otimes a_2)_\lambda ((b_1\circ^{1}_\mu c_1)\otimes (b_2\circ^{2}_\mu c_2))]\\
&=[{a_1}_{\lambda}(b_1\circ^{1}_\mu c_1)]^{1}\otimes (a_2\circ^{2}_\lambda (b_2\circ^{2}_\mu c_2))+(a_1\circ^{1}_\lambda (b_1\circ^{1}_\mu c_1))\otimes [{a_2}_{\lambda}(b_2\circ^{2}_\mu c_2)]^{2},\\
&[(b_1\otimes b_2)_\mu ((a_1\otimes a_2)\circ_\lambda (c_1\otimes c_2))]
=[(b_1\otimes b_2)_\mu ((a_1\circ^{1}_\lambda c_1)\otimes (a_2\circ^{2}_\lambda c_2))]\\
&=[{b_1}_{\mu}(a_1\circ^{1}_\lambda c_1)]^{1}\otimes (b_2\circ^{2}_\mu (a_2\circ^{2}_\lambda c_2))+(b_1\circ^{1}_\mu (a_1\circ^{1}_\lambda c_1))\otimes [{b_2}_{\mu}(a_2\circ^{2}_\lambda c_2)]^{2}.
\end{align*}
Then we have 
\begin{align*}
&\quad \ [(a_1\otimes a_2)_\lambda ((b_1\otimes b_2)\circ_\mu (c_1\otimes c_2))]-[(b_1\otimes b_2)_\mu ((a_1\otimes a_2)\circ_\lambda (c_1\otimes c_2))]\\
&=([{a_1}_{\lambda}(b_1\circ^{1}_\mu c_1)]^{1}-[{b_1}_{\mu}(a_1\circ^{1}_\lambda c_1)]^{1})\otimes ((a_2\circ^{2}_\lambda b_2)\circ^{2}_{\lambda+\mu} c_2)\\
&\hphantom{=~}+((a_1\circ^{1}_\lambda b_1)\circ^{1}_{\lambda+\mu} c_1)\otimes ([{a_2}_{\lambda}(b_2\circ^{2}_\mu c_2)]^{2}-[{b_2}_{\mu}(a_2\circ^{2}_\lambda c_2)]^{2})\\
&=2([{a_1}_{\lambda}b_1]^{1}\circ^1_{\lambda+\mu}c_1)\otimes ((a_2\circ^{2}_\lambda b_2)\circ^{2}_{\lambda+\mu} c_2)
+2((a_1\circ^{1}_\lambda b_1)\circ^{1}_{\lambda+\mu} c_1)\otimes ([{a_2}_{\lambda}b_2]^{2}\circ^2_{\lambda+\mu}c_2)\\
&=2([{a_1}_{\lambda}b_1]^{1}\otimes (a_2\circ^{2}_\lambda b_2))\circ_{\lambda+\mu} (c_1\otimes c_2)
+2((a_1\circ^{1}_\lambda b_1)\otimes [{a_2}_{\lambda}b_2]^{2})\circ_{\lambda+\mu} (c_1\otimes c_2)\\
&=2([{a_1}_{\lambda}b_1]^{1}\otimes (a_2\circ^{2}_\lambda b_2)+(a_1\circ^{1}_\lambda b_1)\otimes [{a_2}_{\lambda}b_2]^{2})\circ_{\lambda+\mu} (c_1\otimes c_2)\\
&=2[(a_1\otimes a_2)_\lambda (b_1\otimes b_2)]\circ_{\lambda+\mu} (c_1\otimes c_2).
\end{align*}

Altogether, $(\mathcal L,\circ_\lambda,[\cdot_\lambda\cdot])$ is a transposed Poisson conformal algebra.
\end{proof}

\subsection{Relation with Hom-Lie conformal algebras}

\begin{proposition}
   Let $(\mathcal{L} ,\circ_\lambda,[\cdot_{\lambda}\cdot])$ be a transposed Poisson conformal algebra. For any $h\in \mathcal{L}$, define a linear map
   $\alpha_h$ : $\mathcal{L}  \rightarrow \mathcal{L}$ by
   \begin{align*}
      \alpha_h(x):=(h\circ_\lambda x)|_{\lambda=0}=h\circ_{(0)} x, \quad \forall ~x\in\mathcal{L}.
   \end{align*}
   Then $(\mathcal{L} ,[\cdot_{\lambda}\cdot],\alpha_h)$ is a Hom-Lie conformal algebra.
\end{proposition}

\begin{proof}
It is well known that $(\mathcal{L},[\cdot_{\lambda}\cdot])$ is a Lie conformal algebra satisfying the axioms of conformal sesquilinearity and skew-symmetry. For all $x\in\mathcal{L}$,
   \begin{align*}
      \alpha_h (\partial x)=(h\circ_\lambda \partial x)|_{\lambda=0}=(\partial+\lambda)(h\circ_\lambda x)|_{\lambda=0}=\partial(h\circ_{(0)} x)=\partial (\alpha_hx).
   \end{align*}
Hence, $\alpha_h\partial=\partial\alpha_h$. By taking $\gamma:=\lambda, ~\mu:=\lambda+\mu,~\lambda:=0$ in Eq.~\eqref{equ}, we obtain
   \begin{align*}
      [[x_{\lambda}y]_{\lambda+\mu}(h\circ_{(0)}z)]+[[y_{\mu}z]_{-\partial-\lambda}(h\circ_{(0)}x)]+[[z_{-\partial-\lambda}x]_{-\partial-\mu}(h\circ_{(0)}y)]=0.
   \end{align*}
That is,
   \begin{align*}
        [[x_{\lambda}y]_{\lambda+\mu}\alpha_h(z)]+[\alpha_h(y)_{\mu}[x_{\lambda}z]]=[\alpha_h(x)_{\lambda}[y_{\mu}z]],
   \end{align*}
which satisfies the axiom of Hom-Jacobi identity. Therefore, $(\mathcal{L} ,[\cdot_{\lambda}\cdot],\alpha_h)$ is a Hom-Lie conformal algebra.
\end{proof}

\subsection{Compatibility conditions between PCA and TPCA}

\begin{proposition}
   Let $(\mathcal{L} ,\circ_\lambda)$ be a commutative associative conformal algebra and $(\mathcal{L} ,[\cdot_{\lambda}\cdot])$ be a Lie conformal algebra.
   Then $(\mathcal{L} ,\circ_\lambda,[\cdot_{\lambda}\cdot])$ is both a Poisson conformal algebra and a transposed Poisson conformal algebra if and only if
   \begin{align*}
      a\circ_{\lambda}[b_{\mu}c]=[b_{\mu-\lambda}a]\circ_{\mu}c=[(a\circ_{\lambda}b)_{\lambda+\mu}c]=0, \quad \forall ~a,b,c\in \mathcal{L}.
   \end{align*}
\end{proposition}

\begin{proof}
   The ``if" part can be easily verified. For the ``only if" part, let $a,b,c\in \mathcal{L}$. By Eq.~\eqref{rule5}, we have
   \begin{align*}
      2[a_{\lambda}b]\circ_{\mu}c=[a_{{\lambda}}(b\circ_{\mu-\lambda}c)]-[b_{\mu-\lambda}(a\circ_{\lambda}c)].
   \end{align*}
   By Eq.~\eqref{rule1}, we obtain
   \begin{align*}
      [a_{\lambda}(b\circ_{\mu-\lambda}c)]=[a_{\lambda}b]\circ_{\mu}c+b\circ_{\mu-\lambda}[a_{\lambda}c],  \quad
      [b_{\mu-\lambda}(a\circ_{\lambda}c)]=[b_{\mu-\lambda}a]\circ_{\mu}c+a\circ_{\lambda}[b_{\mu-\lambda}c].
   \end{align*}
   Thus, we have
   \begin{align*}
      [a_{\lambda}b]\circ_{\mu}c+a\circ_{\lambda}[b_{\mu-\lambda}c]-b\circ_{\mu-\lambda}[a_{\lambda}c]+[b_{\mu-\lambda}a]\circ_{\mu}c=0.
   \end{align*}
   By Eq.~\eqref{eqw}, we get
   \begin{align*}
      [a_{\lambda}b]\circ_{\mu}c+a\circ_{\lambda}[b_{\mu-\lambda}c]-b\circ_{\mu-\lambda}[a_{\lambda}c]=0.
   \end{align*}
   Hence, $[b_{\mu-\lambda}a]\circ_{\mu}c=0$. From this and Eq.~\eqref{eqh}, we get $a\circ_{\lambda}[b_{\mu}c]=b \circ_{\mu} [a _\lambda c]$. By Eq.~\eqref{rule3}, we have
   \begin{align*}
      2a\circ_{\lambda}[b_{\mu}c]=[(a\circ_{\lambda}b)_{\lambda+\mu}c]+[b_{\mu}(a\circ_{\lambda}c)].	
   \end{align*}
   By Eqs.~\eqref{rule1} and \eqref{rule6}, we obtain
   \begin{align*}
      [(a\circ_{\lambda}b)_{\lambda+\mu}c] = b \circ_{\mu} [a _\lambda c] + a \circ_\lambda [b _{\mu} c],  \quad
      [b_{\mu}(a\circ_{\lambda}c)] = [b_{\mu}a]\circ_{\lambda+\mu}c+a\circ_{\lambda}[b_{\mu}c].
   \end{align*}
   Thus, $a\circ_{\lambda}[b_{\mu}c]=-[b_{\mu}a]\circ_{\lambda+\mu}c=0$, $[(a\circ_{\lambda}b)_{\lambda+\mu}c]=2a\circ_{\lambda}[b_{\mu}c]=0$.
\end{proof}

\section{Constructions of transposed Poisson conformal algebras}\label{sec4}
In this section, we first show that a new transposed Poisson conformal algebra can be constructed from a given one by defining an appropriate binary operation.
Then, we introduce the notions of Novikov-Poisson conformal algebras, pre-Lie commutative conformal algebras, differential Novikov-Poisson conformal algebras, and pre-Lie Poisson conformal algebras,
which serve as the conformal analogues of Novikov-Poisson algebras \cite{Xu97}, pre-Lie commutative algebras \cite{MA14}, differential Novikov-Poisson algebras \cite{BCZ18}, and pre-Lie Poisson algebras \cite{Bai23}, respectively.
In addition, we present several constructions of TPCA arising from these algebraic structures.

\begin{proposition}
   Let $(\mathcal{L} ,\circ_\lambda,[\cdot_{\lambda}\cdot])$ be a transposed Poisson conformal algebra. For any $h\in \mathcal{L}$, define a new binary operation $[\cdot_{\lambda}\cdot]^h$ on $\mathcal{L}$ by
   \begin{align*}
      [x_{\lambda}y]^h:=(h\circ_{\mu}[x_{\lambda}y])|_{\mu=0}=h\circ_{(0)}[x_{\lambda}y], \quad \forall ~x,y\in \mathcal{L}.
   \end{align*}
   Then $(\mathcal{L} ,\circ_\lambda, [\cdot_{\lambda}\cdot]^h)$ is a transposed Poisson conformal algebra.
\end{proposition}

\begin{proof}
Let $x,y,z,h\in \mathcal{L}$. We first show that $(\mathcal{L}, [\cdot_{\lambda}\cdot]^h)$ is a Lie conformal algebra.

   (a) {\bf Conformal sesquilinearity:}
   \begin{align*}
      [(\partial x)_{\lambda}y]^h=h\circ_{(0)}[(\partial x)_{\lambda}y]=h\circ_{(0)}(-\lambda[x_{\lambda}y])=-\lambda (h\circ_{(0)}[x_{\lambda}y])=-\lambda [x_{\lambda}y]^h.
   \end{align*}

   (b) {\bf Skew-symmetry:}
   \begin{align*}
      [x_{\lambda}y]^h=h\circ_{(0)}[x_{\lambda}y]=h\circ_{(0)}(-[y_{-\partial-\lambda}x])=-h\circ_{(0)}[y_{-\partial-\lambda}x]=-[y_{-\partial-\lambda}x]^h.
   \end{align*}

   (c) {\bf Jacobi identity:} By taking $\gamma:=\lambda, ~\mu:=\lambda+\mu,~\lambda:=0$ in Eq.~\eqref{eqs}, we obtain
   \begin{align*}
      &[(h\circ_{(0)}[x_{\lambda}y])_{\lambda+\mu}z]+[(h\circ_{(0)}[y_{\mu}z])_{-\partial-\lambda}x]+[(h\circ_{(0)}[z_{-\partial-\lambda}x])_{-\partial-\mu}y]=0.
   \end{align*}
   Thus, we have
   $$h\circ_{(0)}([(h\circ_{(0)}[x_{\lambda}y])_{\lambda+\mu}z]+[(h\circ_{(0)}[y_{\mu}z])_{-\partial-\lambda}x]+[(h\circ_{(0)}[z_{-\partial-\lambda}x])_{-\partial-\mu}y])=0.$$
    That is,
   \begin{align*}
      [{[x_{\lambda}y]^h}_{\lambda+\mu}z]^h+[{[y_{\mu}z]^h}_{-\partial-\lambda}x]^h+[{[z_{-\partial-\lambda}x]^h}_{-\partial-\mu}y]^h=0.
   \end{align*}

   Next we show that transposed conformal Leibniz rule holds, by Eq.~\eqref{rule5}, we obtain
     \begin{align*}
      & 2([x_{\lambda}y]^h\circ_{\mu}z)=2((h\circ_{(0)}[x_{\lambda}y])\circ_{\mu}z)=2(h\circ_{(0)}([x_{\lambda}y]\circ_{\mu}z))
      =h\circ_{(0)}(2([x_{\lambda}y]\circ_{\mu}z))\\
      &= h\circ_{(0)}([x_{{\lambda}}(y\circ_{\mu-\lambda}z)]-[y_{\mu-\lambda}(x\circ_{\lambda}z)])
      =h\circ_{(0)}[x_{{\lambda}}(y\circ_{\mu-\lambda}z)]-h\circ_{(0)}[y_{\mu-\lambda}(x\circ_{\lambda}z)]  \\
      &= [x_{{\lambda}}(y\circ_{\mu-\lambda}z)]^h-[y_{\mu-\lambda}(x\circ_{\lambda}z)]^h.
     \end{align*}
Therefore, $(\mathcal{L} ,\circ_\lambda, [\cdot_{\lambda}\cdot]^h)$ is a transposed Poisson conformal algebra.
\end{proof}

\begin{definition}
   A {\textbf{Novikov-Poisson conformal algebra}} $(\mathcal{N},\circ_{\lambda},\ast_{\lambda})$ is a $\mathbb{C}[\partial]$-module $\mathcal{N}$ endowed with a $\mathbb{C}$-bilinear map $\mathcal{N}  \otimes \mathcal{N}  \rightarrow \mathcal{N} [\lambda]$, denoted by $a\otimes b\mapsto a \circ_{\lambda}b$ and a $\mathbb{C}$-bilinear map $\mathcal{N}  \otimes \mathcal{N}  \rightarrow \mathcal{N} [\lambda]$,
   denoted by $a\otimes b\mapsto a\ast_\lambda b$ such that $(\mathcal{N} ,\circ_\lambda)$ is a {commutative associative conformal algebra}, $(\mathcal{N} ,\ast_\lambda)$ is a {Novikov conformal algebra} and the following equations hold for all $a,b,c \in \mathcal{N}$:
   \begin{align}
      &(a\circ_\lambda b)\ast_{\lambda+\mu}c=a\circ_\lambda(b\ast_\mu c), \label{f7}\\
      &(a\ast_\lambda b)\circ_{\lambda+\mu}c-a\ast_\lambda(b\circ_\mu c)=(b\ast_\mu a)\circ_{\lambda+\mu}c-b\ast_\mu(a\circ_\lambda c). \label{h8}
   \end{align}
\end{definition}

\begin{remark}\label{rem2}
   In analogy with Remark \ref{rem1}, the following properties always hold in a Novikov-Poisson conformal algebra $(\mathcal{N},\circ_{\lambda},\ast_{\lambda})$:
   \begin{align*}
         a\circ_{\lambda}(b\ast_{-\partial-\mu}c) & =(a\circ_{\lambda}b)\ast_{-\partial-\mu}c, \\
         a\circ_{-\partial-\lambda}(b\ast_{\mu}c) & =(a\circ_{-\partial-\mu}b)\ast_{-\partial+\mu-\lambda}c, \\
         a\circ_{-\partial-\lambda}(b\ast_{-\partial-\mu}c) & =(a\circ_{-\partial+\mu-\lambda}b)\ast_{-\partial-\mu}c.
   \end{align*}
\end{remark}

Taking the commutator in a Novikov-Poisson conformal algebra, we obtain
\begin{theorem}\label{The1.2}
   Let $(\mathcal{N} ,\circ_\lambda,\ast_\lambda)$ be a Novikov-Poisson conformal algebra. Define
   \begin{align}
      [a_{\lambda}b]=a\ast_{\lambda} b-b\ast_{-\partial-\lambda} a, \quad \forall ~a,b \in \mathcal{N}.
   \end{align}
Then $(\mathcal{N} ,\circ_\lambda,[\cdot_\lambda\cdot])$ is a transposed Poisson conformal algebra.
\end{theorem}

\begin{proof}
   It is well known that $(\mathcal{N},[\cdot_\lambda\cdot])$ is a Lie conformal algebra (see Proposition \ref{propo1}). For all $a,b,c\in \mathcal{N}$, we have
   \begin{align*}
      &[a_{\lambda}b]\circ_{\mu}c=(a\ast_{\lambda} b-b\ast_{-\partial-\lambda} a)\circ_{\mu}c=(a\ast_{\lambda} b)\circ_{\mu}c-(b\ast_{-\partial-\lambda} a)\circ_{\mu}c, \\
      &[a_{{\lambda}}(b\circ_{\mu-\lambda}c)]=a\ast_{\lambda} (b\circ_{\mu-\lambda}c)-(b\circ_{\mu-\lambda}c)\ast_{-\partial-\lambda} a, \\
      &[b_{\mu-\lambda}(a\circ_{\lambda}c)]=b\ast_{\mu-\lambda} (a\circ_{\lambda}c)-(a\circ_{\lambda}c)\ast_{-\partial-\mu+\lambda} b.
   \end{align*}
   By Remark \ref{rem2}, we obtain
   \begin{align*}
      &(a\ast_{\lambda} b)\circ_{\mu}c=c\circ_{-\partial-\mu}(a\ast_{\lambda} b)=(c\circ_{-\partial-\lambda}a)\ast_{-\partial+\lambda-\mu}b=(a\circ_{\lambda}c)\ast_{-\partial-\mu+\lambda} b, \\
      &(b\ast_{-\partial-\lambda} a)\circ_{\mu}c=c\circ_{-\partial-\mu}(b\ast_{-\partial-\lambda} a)=(c\circ_{-\partial+\lambda-\mu}b)\ast_{-\partial-\lambda}a=(b\circ_{\mu-\lambda}c)\ast_{-\partial-\lambda}a.
   \end{align*}
   Thus, we have
   \begin{align*}
      [a_{{\lambda}}(b\circ_{\mu-\lambda}c)]-[b_{\mu-\lambda}(a\circ_{\lambda}c)]=a\ast_{\lambda} (b\circ_{\mu-\lambda}c)-b\ast_{\mu-\lambda} (a\circ_{\lambda}c)+[a_{\lambda}b]\circ_{\mu}c.
   \end{align*}
   By taking $\lambda:=\lambda,~\mu:=\mu-\lambda$ in Eq.~\eqref{h8}, we get
   \begin{align*}
      (a\ast_\lambda b)\circ_{\mu}c-a\ast_\lambda(b\circ_{\mu-\lambda} c)=(b\ast_{\mu-\lambda} a)\circ_{\mu}c-b\ast_{\mu-\lambda}(a\circ_\lambda c).
   \end{align*}
   Since $(b\ast_{-\partial-\lambda} a)\circ_{\mu}c=(b\ast_{\mu-\lambda} a)\circ_{\mu}c$, it follows further that
   $$(a\ast_\lambda b)\circ_{\mu}c-(b\ast_{\mu-\lambda} a)\circ_{\mu}c=a\ast_\lambda(b\circ_{\mu-\lambda} c)-b\ast_{\mu-\lambda}(a\circ_\lambda c)=[a_{\lambda}b]\circ_{\mu}c.$$
   Combining the above equalities, we arrive at
   \begin{align*}
      [a_{{\lambda}}(b\circ_{\mu-\lambda}c)]-[b_{\mu-\lambda}(a\circ_{\lambda}c)]=2[a_{\lambda}b]\circ_{\mu}c,
   \end{align*}
   it follows that relation \eqref{rule5} holds. Hence the conclusion holds.
\end{proof}

\begin{lemma}\label{Lem1.1}
   Let $(\mathcal{N} ,\circ_\lambda)$ be a commutative associative conformal algebra and $D$ be a derivation. Define a binary operation $\ast_\lambda$ on $\mathcal{N}$ by
   \begin{align}
      a\ast_\lambda b=a\circ_\lambda D(b), \quad \forall~a,b \in \mathcal{N}.
   \end{align}
Then the following conclusions hold:

(1) $(\mathcal{N} ,\ast_\lambda)$ is a Novikov conformal algebra.

(2) $(\mathcal{N} ,\circ_\lambda,\ast_\lambda)$ is a Novikov-Poisson conformal algebra.
\end{lemma}

\begin{proof}
(1) Let $a,b,c \in \mathcal{N}$. (a) We first show that $(\mathcal{N} ,\ast_\lambda)$ is a conformal algebra.
   \begin{align*}
      &(\partial a)\ast_\lambda b=(\partial a)\circ_\lambda D(b)=-\lambda (a\circ_\lambda D(b))=-\lambda(a\ast_\lambda b), \\
      &a\ast_\lambda (\partial b)=a\circ_\lambda D(\partial b)=a\circ_\lambda (\partial D(b))=(\partial+\lambda) (a\circ_\lambda D(b))=(\partial+\lambda)(a\ast_\lambda b).
   \end{align*}
(b) Next we show that $(\mathcal{N} ,\ast_\lambda)$ is a left-symmetric conformal algebra. Since
   \begin{align*}
      &(a\ast_\lambda b)\ast_{\lambda+\mu}c-a\ast_\lambda(b\ast_\mu c)=(a\circ_\lambda D(b))\circ_{\lambda+\mu} D(c)-a\circ_\lambda D(b\circ_\mu D(c))  \\
      &=a\circ_\lambda (D(b)\circ_{\mu} D(c))-a\circ_\lambda (D(b)\circ_\mu D(c))-a\circ_\lambda (b\circ_\mu D^2(c))=-a\circ_\lambda (b\circ_\mu D^2(c)),
   \end{align*}
we similarly have
   \begin{align*}
        (b\ast_\mu a)\ast_{\lambda+\mu}c-b\ast_\mu(a\ast_\lambda c)=-b\circ_\mu (a\circ_\lambda D^2(c)).
   \end{align*}
By Remark \ref{rem3}, we obtain $a\circ_\lambda (b\circ_\mu D^2(c))=b\circ_\mu (a\circ_\lambda D^2(c)).$ Thus, we get
$$(a\ast_\lambda b)\ast_{\lambda+\mu}c-a\ast_\lambda(b\ast_\mu c)=(b\ast_\mu a)\ast_{\lambda+\mu}c-b\ast_\mu(a\ast_\lambda c).$$
(c) Finally, we show that $(\mathcal{N} ,\ast_\lambda)$ is a Novikov conformal algebra. Since
   \begin{align*}
      (a\ast_\lambda c)\ast_{-\mu-\partial}b&=(a\circ_\lambda D(c))\circ_{-\mu-\partial}D(b)=D(b)\circ_{\mu}(a\circ_\lambda D(c)) \\
      &=a\circ_\lambda(D(b)\circ_{\mu} D(c))=(a\circ_\lambda D(b))\circ_{\lambda+\mu} D(c)=(a\ast_\lambda b)\ast_{\lambda+\mu}c.
   \end{align*}
Hence, $(\mathcal{N} ,\ast_\lambda)$ is a Novikov conformal algebra.

(2) Let $a,b,c \in \mathcal{N}$. (d) Relation \eqref{f7} holds. Since
   \begin{align*}
      (a\circ_\lambda b)\ast_{\lambda+\mu}c=(a\circ_\lambda b)\circ_{\lambda+\mu}D(c)=a\circ_\lambda(b\circ_{\mu}D(c))=a\circ_\lambda(b\ast_{\mu}c).
   \end{align*}
(e) Relation \eqref{h8} holds. Since
   \begin{align*}
      &(a\ast_\lambda b)\circ_{\lambda+\mu}c-a\ast_\lambda(b\circ_\mu c)=(a\circ_\lambda D(b))\circ_{\lambda+\mu}c-a\circ_\lambda D(b\circ_\mu c) \\
      &=a\circ_\lambda (D(b)\circ_{\mu}c)-a\circ_\lambda (D(b)\circ_{\mu}c)-a\circ_\lambda (b\circ_{\mu}D(c))=-a\circ_\lambda (b\circ_{\mu}D(c)),
   \end{align*}
we similarly have
 $$(b\ast_\mu a)\circ_{\lambda+\mu}c-b\ast_\mu(a\circ_\lambda c)=-b\circ_\mu (a\circ_{\lambda}D(c)).$$
By Remark \ref{rem3}, we get $a\circ_\lambda (b\circ_{\mu}D(c))=b\circ_\mu (a\circ_{\lambda}D(c))$. Thus, we obtain
   \begin{align*}
      (a\ast_\lambda b)\circ_{\lambda+\mu}c-a\ast_\lambda(b\circ_\mu c)=(b\ast_\mu a)\circ_{\lambda+\mu}c-b\ast_\mu(a\circ_\lambda c).
   \end{align*}
Altogether, $(\mathcal{N} ,\circ_\lambda,\ast_\lambda)$ is a Novikov-Poisson conformal algebra.
\end{proof}

Combining Theorem \ref{The1.2} and Lemma \ref{Lem1.1}, we obtain the following corollary.

\begin{corollary}
   Let $(\mathcal{N} ,\circ_\lambda)$ be a commutative associative conformal algebra and $D$ be a derivation. Then $(\mathcal{N} ,\circ_\lambda,[\cdot_{\lambda}\cdot])$ is a transposed Poisson conformal algebra, where
   \begin{align}
      [a_{\lambda}b]=a\circ_\lambda D(b)-b\circ_{-\partial-\lambda} D(a), \quad \forall ~a,b \in \mathcal{N}.
   \end{align}
\end{corollary}

\begin{definition}
   A {\bf pre-Lie commutative conformal algebra} $(\mathcal{L} ,\circ_\lambda,\ast_\lambda)$ is a $\mathbb{C}[\partial]$-module $\mathcal{L}$ endowed with a $\mathbb{C}$-bilinear map $\mathcal{L}  \otimes \mathcal{L}  \rightarrow \mathcal{L} [\lambda]$, denoted by $a\otimes b\mapsto a \circ_{\lambda}b$ and a $\mathbb{C}$-bilinear map $\mathcal{L}  \otimes \mathcal{L}  \rightarrow \mathcal{L} [\lambda]$,
   denoted by $a\otimes b\mapsto a \ast_\lambda b$ such that $(\mathcal{L} ,\circ_\lambda)$ is a {commutative associative conformal algebra}, $(\mathcal{L} ,\ast_\lambda)$ is a {left-symmetric conformal algebra} and the following equation holds for all $a,b,c \in \mathcal{L}$:
   \begin{align}\label{hh1}
      a\ast_\lambda(b\circ_\mu c)=(a\ast_\lambda b)\circ_{\lambda+\mu}c+b\circ_\mu (a\ast_\lambda c).
   \end{align}
\end{definition}

\begin{example}
   Let $(\mathcal{L} ,\circ_\lambda)$ be a commutative associative conformal algebra and $D$ be a derivation.
   It is straightforward to verify that $(\mathcal{L} ,\circ_\lambda,\ast_\lambda)$ is a pre-Lie commutative conformal algebra,
   where $(\mathcal{L} ,\ast_\lambda)$ is defined by $a\ast_\lambda b=a\circ_\lambda D(b)$ for all $a,b \in \mathcal{L}$.
\end{example}

A {\bf differential Novikov-Poisson conformal algebra} is defined as a triple $(\mathcal{N} ,\circ_\lambda,\ast_\lambda)$, where
$(\mathcal{N} ,\circ_\lambda)$ is a {commutative associative conformal algebra} and $(\mathcal{N} ,\ast_\lambda)$ is a {Novikov conformal algebra}
satisfying Eqs.~\eqref{f7}, \eqref{h8} and \eqref{hh1}.

\begin{definition}
   A {\textbf{pre-Lie Poisson conformal algebra}} $(\mathcal{L} ,\circ_\lambda,\ast_\lambda)$ is a $\mathbb{C}[\partial]$-module $\mathcal{L} $ endowed with a $\mathbb{C}$-bilinear map $\mathcal{L}  \otimes \mathcal{L}  \rightarrow \mathcal{L} [\lambda]$, denoted by $a\otimes b\mapsto a \circ_{\lambda}b$ and a $\mathbb{C}$-bilinear map $\mathcal{L}  \otimes \mathcal{L}  \rightarrow \mathcal{L} [\lambda]$,
   denoted by $a\otimes b\mapsto a \ast_\lambda b$ such that $(\mathcal{L} ,\circ_\lambda)$ is a {commutative associative conformal algebra}, $(\mathcal{L} ,\ast_\lambda)$ is a {left-symmetric conformal algebra} and the following equations hold for all $a,b,c \in \mathcal{L}$:
   \begin{align}
      &(a\circ_\lambda b)\ast_{\lambda+\mu}c=a\circ_\lambda(b\ast_\mu c),\\
      &(a\ast_\lambda b)\circ_{\lambda+\mu}c-a\ast_\lambda(b\circ_\mu c)=(b\ast_\mu a)\circ_{\lambda+\mu}c-b\ast_\mu(a\circ_\lambda c).
   \end{align}
\end{definition}

\begin{lemma}\label{Lem1.2}
The relationship among these notions can be described as follows:

(1) A Novikov-Poisson conformal algebra is a pre-Lie Poisson conformal algebra.

(2) A pre-Lie commutative conformal algebra $(\mathcal{L} ,\circ_\lambda,\ast_\lambda)$ satisfying Eq.~\eqref{f7} is a pre-Lie Poisson conformal algebra.

(3) A differential Novikov-Poisson conformal algebra is a pre-Lie commutative conformal algebra satisfying Eq.~\eqref{f7} and hence is pre-Lie Poisson conformal algebra.
\end{lemma}

\begin{proof}
    The proofs of (1) and (3) are straightforward. It suffices to prove (2). Let $a,b,c \in \mathcal{L}$. By Eqs.~\eqref{f7} and \eqref{hh1}, we have
   \begin{align*}
        &(a\ast_\lambda b)\circ_{\lambda+\mu}c-a\ast_\lambda(b\circ_\mu c)=-b\circ_\mu (a\ast_\lambda c)=-(b\circ_{\mu} a)\ast_{\lambda+\mu}c,\\
        &(b\ast_\mu a)\circ_{\lambda+\mu}c-b\ast_\mu(a\circ_\lambda c)=-a\circ_\lambda (b\ast_\mu c)=-(a\circ_\lambda b)\ast_{\lambda+\mu}c.
   \end{align*}
Observe that
   \begin{align*}
        (a\circ_\lambda b)\ast_{\lambda+\mu}c=(b\circ_{-\partial-\lambda} a)\ast_{\lambda+\mu}c=(b\circ_{\mu} a)\ast_{\lambda+\mu}c.
   \end{align*}
Thus, we obtain
   \begin{align*}
       (a\ast_\lambda b)\circ_{\lambda+\mu}c-a\ast_\lambda(b\circ_\mu c)=(b\ast_\mu a)\circ_{\lambda+\mu}c-b\ast_\mu(a\circ_\lambda c).
   \end{align*}
Hence, $(\mathcal{L} ,\circ_\lambda,\ast_\lambda)$ is a pre-Lie Poisson conformal algebra.
\end{proof}

The commutator of a left-symmetric conformal algebra is a Lie conformal algebra (Proposition \ref{propo1}).
By applying a similar argument as in the proof of Theorem \ref{The1.2}, we obtain the following proposition.

\begin{proposition}\label{Pro2}
   Let $(\mathcal{L} ,\circ_\lambda,\ast_\lambda)$ be a pre-Lie Poisson conformal algebra. Define
   \begin{align}
      [a_{\lambda}b]=a\ast_{\lambda} b-b\ast_{-\partial-\lambda} a, \quad \forall ~a,b \in \mathcal{L}.
   \end{align}
Then $(\mathcal{L} ,\circ_\lambda,[\cdot_\lambda\cdot])$ is a transposed Poisson conformal algebra.
\end{proposition}

Combining Proposition \ref{Pro2} and Lemma \ref{Lem1.2}, we obtain the following corollary.
\begin{corollary}
   Let $(\mathcal{L} ,\circ_\lambda,\ast_\lambda)$ be a pre-Lie commutative conformal algebra. Suppose that Eq.~\eqref{f7} holds. Then $(\mathcal{L} ,\circ_\lambda,[\cdot_\lambda\cdot])$ is a transposed Poisson conformal algebra, where
   the binary operation $[\cdot_\lambda\cdot]$ is defined by $[a_{\lambda}b]=a\ast_{\lambda} b-b\ast_{-\partial-\lambda} a$ for all $a,b \in \mathcal{L}$. In particular, the conclusion holds when $(\mathcal{L} ,\circ_\lambda,\ast_\lambda)$ is a differential Novikov-Poisson conformal algebra.
\end{corollary}


\section{Compatible noncommutative TPCA structures on the Lie conformal algebra $W(a,b)$}\label{sec5}
We begin this section with the following definition.
\begin{definition}
Let $(A,[{\cdot_ \lambda \cdot}])$ be a {Lie conformal algebra}. A (noncommutative) TPCA structure over $A$ is a binary $\lambda$-multiplication $\circ_{\lambda}$ on $A$, such that $(A,\circ_{\lambda},[{\cdot_ \lambda \cdot}])$ forms a (noncommutative) TPCA.
\end{definition}

\begin{proposition}
The compatible (noncommutative) TPCA structures over the Virasoro Lie conformal algebra $Vir$ are of the form:
   \begin{align}
        Vir=\mathbb{C}[\partial]L,\quad L\circ_\lambda L=cL,
   \end{align}
where $c$ is a complex number. 
\end{proposition}

\begin{proof}
Eqs.~\eqref{ly} and \eqref{rule3} give the following relations:
   \begin{align}
      &L\circ_{\lambda}\left(L\circ_{\mu} L\right)=\left(L\circ_{\lambda} L\right)\circ_{\lambda+\mu} L,\label{ly1}\\
      &{2} (L\circ_{\lambda}[L_{\mu}L])=[(L\circ_{\lambda}L)_{\lambda+\mu}L]+[L_{\mu}(L\circ_{\lambda}L)].\label{ly2}
   \end{align}
Let us set $L\circ_\lambda L=f(\partial, \lambda)L$ for some polynomial $f$. Then Eq.~\eqref{ly1} is equivalent to
    \begin{align}\label{ly3}
        f(\partial+\lambda, \mu) f(\partial, \lambda)=f(-\lambda-\mu, \lambda) f(\partial, \lambda+\mu),
    \end{align}
which implies $deg_{\partial} f(\partial, \lambda)+deg_{\partial} f(\partial, \lambda)=deg_{\partial} f(\partial, \lambda)$,
where we use the notation $deg_{\partial} f(\partial, \lambda)$ to denote the highest degree of $\partial$ in $f(\partial, \lambda)$. Hence, $deg_{\partial} f(\partial, \lambda)=0$.
Now we can suppose $f(\partial, \lambda)=f(\lambda)$ for some $f(\lambda)\in \mathbb{C}[\lambda]$. Then Eq.~\eqref{ly3} reduces to 
    \begin{align}
        f(\mu) f(\lambda)=f(\lambda) f(\lambda+\mu),
    \end{align}
which implies $f(\lambda)=c$ for some $c\in \mathbb{C}$. Moreover, it is easy to check that Eq.~\eqref{ly2} holds and $L\circ_\lambda L=L\circ_{-\partial-\lambda} L$. Hence the conclusion holds.
\end{proof}

Next we turn to the classification of compatible noncommutative TPCA structures over the Lie conformal algebra $W(a,b)$.
By the definition of noncommutative TPCA, we can assume
\begin{alignat}{2}
L\circ_\lambda L&=f_1(\partial, \lambda)L+f_2(\partial, \lambda)M,&\quad L\circ_\lambda M&=g_1(\partial, \lambda)L+g_2(\partial, \lambda)M,\label{ly4}\\
M\circ_\lambda L&=h_1(\partial, \lambda)L+h_2(\partial, \lambda)M,&\quad M\circ_\lambda M&=q_1(\partial, \lambda)L+q_2(\partial, \lambda)M,\label{ly5}
\end{alignat}
where $f_1(\partial, \lambda)$, $f_2(\partial, \lambda)$, $g_1(\partial, \lambda)$, $g_2(\partial, \lambda)$,
$h_1(\partial, \lambda)$, $h_2(\partial, \lambda)$, $q_1(\partial, \lambda)$, and $q_2(\partial, \lambda)$ are polynomials in $\mathbb{C}[\partial,\lambda]$.

By the associative conformal identity \eqref{ly} and the transposed conformal Leibniz rule \eqref{rule3}, the algebra defined by Eqs.~\eqref{ly4} and \eqref{ly5} is a compatible noncommutative TPCA structure on $W(a,b)$ if and only if the following identities hold:
\begin{alignat*}{2}
    &L\circ_\lambda (L\circ_\mu L) = (L\circ_{\lambda} L) \circ_{\lambda+\mu} L, &\quad
    &2L\circ_{\lambda}[L_{\mu}L] = [(L\circ_{\lambda}L)_{\lambda+\mu}L] + [L_{\mu}(L\circ_{\lambda}L)], \\
    &L\circ_\lambda (L\circ_\mu M) = (L\circ_{\lambda} L) \circ_{\lambda+\mu} M, &
    &2L\circ_{\lambda}[L_{\mu}M] = [(L\circ_{\lambda}L)_{\lambda+\mu}M] + [L_{\mu}(L\circ_{\lambda}M)], \\
    &L\circ_\lambda (M\circ_\mu L) = (L\circ_{\lambda} M) \circ_{\lambda+\mu} L, &
    &2L\circ_{\lambda}[M_{\mu}L] = [(L\circ_{\lambda}M)_{\lambda+\mu}L] + [M_{\mu}(L\circ_{\lambda}L)], \\
    &L\circ_\lambda (M\circ_\mu M) = (L\circ_{\lambda} M) \circ_{\lambda+\mu} M, &
    &2L\circ_{\lambda}[M_{\mu}M] = [(L\circ_{\lambda}M)_{\lambda+\mu}M] + [M_{\mu}(L\circ_{\lambda}M)], \\
    &M\circ_\lambda (L\circ_\mu L) = (M\circ_{\lambda} L) \circ_{\lambda+\mu} L, &
    &2M\circ_{\lambda}[L_{\mu}L] = [(M\circ_{\lambda}L)_{\lambda+\mu}L] + [L_{\mu}(M\circ_{\lambda}L)], \\
    &M\circ_\lambda (L\circ_\mu M) = (M\circ_{\lambda} L) \circ_{\lambda+\mu} M, &
    &2M\circ_{\lambda}[L_{\mu}M] = [(M\circ_{\lambda}L)_{\lambda+\mu}M] + [L_{\mu}(M\circ_{\lambda}M)], \\
    &M\circ_\lambda (M\circ_\mu L) = (M\circ_{\lambda} M) \circ_{\lambda+\mu} L, &
    &2M\circ_{\lambda}[M_{\mu}L] = [(M\circ_{\lambda}M)_{\lambda+\mu}L] + [M_{\mu}(M\circ_{\lambda}L)], \\
    &M\circ_\lambda (M\circ_\mu M) = (M\circ_{\lambda} M) \circ_{\lambda+\mu} M, &
    &2M\circ_{\lambda}[M_{\mu}M] = [(M\circ_{\lambda}M)_{\lambda+\mu}M] + [M_{\mu}(M\circ_{\lambda}M)].
\end{alignat*}

Substituting Eqs.~\eqref{ly4} and \eqref{ly5} into these identities above, the algebra defined by Eqs.~\eqref{ly4} and \eqref{ly5} is a compatible noncommutative TPCA structure on $W(a,b)$ if and only if
$f_1(\partial, \lambda)$, $f_2(\partial, \lambda)$, $g_1(\partial, \lambda)$, $g_2(\partial, \lambda)$,
$h_1(\partial, \lambda)$, $h_2(\partial, \lambda)$, $q_1(\partial, \lambda)$, and $q_2(\partial, \lambda)$ satisfy the following equations:
\begin{align}
f_1(\partial+\lambda, \mu)f_1(\partial, \lambda)+f_2(\partial&+\lambda, \mu)g_1(\partial, \lambda) \nonumber\\
&=f_1(-\lambda-\mu, \lambda)f_1(\partial, \lambda+\mu)+f_2(-\lambda-\mu, \lambda)h_1(\partial, \lambda+\mu),\label{yyy8}\\
f_1(\partial+\lambda, \mu)f_2(\partial, \lambda)+f_2(\partial&+\lambda, \mu)g_2(\partial, \lambda) \nonumber\\
&=f_1(-\lambda-\mu, \lambda)f_2(\partial, \lambda+\mu)+f_2(-\lambda-\mu, \lambda)h_2(\partial, \lambda+\mu),\label{yyy9}\\
g_1(\partial+\lambda, \mu)f_1(\partial, \lambda)+g_2(\partial&+\lambda, \mu)g_1(\partial, \lambda)\nonumber\\
&=f_1(-\lambda-\mu, \lambda)g_1(\partial, \lambda+\mu)+f_2(-\lambda-\mu, \lambda)q_1(\partial, \lambda+\mu),\label{yyy10}\\
g_1(\partial+\lambda, \mu)f_2(\partial, \lambda)+g_2(\partial&+\lambda, \mu)g_2(\partial, \lambda)\nonumber\\
&=f_1(-\lambda-\mu, \lambda)g_2(\partial, \lambda+\mu)+f_2(-\lambda-\mu, \lambda)q_2(\partial, \lambda+\mu),\label{yyy11}\\
h_1(\partial+\lambda, \mu)f_1(\partial, \lambda)+h_2(\partial&+\lambda, \mu)g_1(\partial, \lambda)\nonumber\\
&=g_1(-\lambda-\mu, \lambda)f_1(\partial, \lambda+\mu)+g_2(-\lambda-\mu, \lambda)h_1(\partial, \lambda+\mu),\label{yyy12}\\
h_1(\partial+\lambda, \mu)f_2(\partial, \lambda)+h_2(\partial&+\lambda, \mu)g_2(\partial, \lambda)\nonumber\\
&=g_1(-\lambda-\mu, \lambda)f_2(\partial, \lambda+\mu)+g_2(-\lambda-\mu, \lambda)h_2(\partial, \lambda+\mu),\label{yyy13}\\
q_1(\partial+\lambda, \mu)f_1(\partial, \lambda)+q_2(\partial&+\lambda, \mu)g_1(\partial, \lambda)\nonumber\\
&=g_1(-\lambda-\mu, \lambda)g_1(\partial, \lambda+\mu)+g_2(-\lambda-\mu, \lambda)q_1(\partial, \lambda+\mu),\label{yyy14}\\
q_1(\partial+\lambda, \mu)f_2(\partial, \lambda)+q_2(\partial&+\lambda, \mu)g_2(\partial, \lambda)\nonumber\\
&=g_1(-\lambda-\mu, \lambda)g_2(\partial, \lambda+\mu)+g_2(-\lambda-\mu, \lambda)q_2(\partial, \lambda+\mu),\label{yyy15}\\
f_1(\partial+\lambda, \mu)h_1(\partial, \lambda)+f_2(\partial&+\lambda, \mu)q_1(\partial, \lambda)\nonumber\\
&=h_1(-\lambda-\mu, \lambda)f_1(\partial, \lambda+\mu)+h_2(-\lambda-\mu, \lambda)h_1(\partial, \lambda+\mu),\label{yyy16}\\
f_1(\partial+\lambda, \mu)h_2(\partial, \lambda)+f_2(\partial&+\lambda, \mu)q_2(\partial, \lambda)\nonumber\\
&=h_1(-\lambda-\mu, \lambda)f_2(\partial, \lambda+\mu)+h_2(-\lambda-\mu, \lambda)h_2(\partial, \lambda+\mu),\label{yyy17}\\
g_1(\partial+\lambda, \mu)h_1(\partial, \lambda)+g_2(\partial&+\lambda, \mu)q_1(\partial, \lambda)\nonumber\\
&=h_1(-\lambda-\mu, \lambda)g_1(\partial, \lambda+\mu)+h_2(-\lambda-\mu, \lambda)q_1(\partial, \lambda+\mu),\label{yyy18}\\
g_1(\partial+\lambda, \mu)h_2(\partial, \lambda)+g_2(\partial&+\lambda, \mu)q_2(\partial, \lambda)\nonumber\\
&=h_1(-\lambda-\mu, \lambda)g_2(\partial, \lambda+\mu)+h_2(-\lambda-\mu, \lambda)q_2(\partial, \lambda+\mu),\label{yyy19}\\
h_1(\partial+\lambda, \mu)h_1(\partial, \lambda)+h_2(\partial&+\lambda, \mu)q_1(\partial, \lambda)\nonumber\\
&=q_1(-\lambda-\mu, \lambda)f_1(\partial, \lambda+\mu)+q_2(-\lambda-\mu, \lambda)h_1(\partial, \lambda+\mu),\label{yyy20}\\
h_1(\partial+\lambda, \mu)h_2(\partial, \lambda)+h_2(\partial&+\lambda, \mu)q_2(\partial, \lambda)\nonumber\\
&=q_1(-\lambda-\mu, \lambda)f_2(\partial, \lambda+\mu)+q_2(-\lambda-\mu, \lambda)h_2(\partial, \lambda+\mu),\label{yyy21}\\
q_1(\partial+\lambda, \mu)h_1(\partial, \lambda)+q_2(\partial&+\lambda, \mu)q_1(\partial, \lambda)\nonumber\\
&=q_1(-\lambda-\mu, \lambda)g_1(\partial, \lambda+\mu)+q_2(-\lambda-\mu, \lambda)q_1(\partial, \lambda+\mu),\label{yyy22}\\
q_1(\partial+\lambda, \mu)h_2(\partial, \lambda)+q_2(\partial&+\lambda, \mu)q_2(\partial, \lambda)\nonumber\\
&=q_1(-\lambda-\mu, \lambda)g_2(\partial, \lambda+\mu)+q_2(-\lambda-\mu, \lambda)q_2(\partial, \lambda+\mu),\label{yyy23}\\
2(\partial+\lambda+2\mu)f_1(\partial, \lambda)&=(\partial+2\lambda+2\mu)f_1(-\lambda-\mu, \lambda)+(\partial+2\mu)f_1(\partial+\mu, \lambda),\label{yyy24}\\
2(\partial+\lambda+2\mu)f_2(\partial, \lambda)&=((a-1)\partial+a{(\lambda+\mu)}-b)f_2(-\lambda-\mu, \lambda)\nonumber\\
&\hphantom{=((a-1)\partial+a{(\lambda+\mu)}-b)} +(\partial+a\mu+b)f_2(\partial+\mu, \lambda),\label{yyy25}\\
2(\partial+\lambda+a\mu+b)g_1(\partial, \lambda)&=(\partial+2\mu)g_1(\partial+\mu, \lambda),\label{yyy26}\\
2(\partial+\lambda+a\mu+b)g_2(\partial, \lambda)&=(\partial+a(\lambda+\mu)+b)f_1(-\lambda-\mu, \lambda)+(\partial+a\mu+b)g_2(\partial+\mu, \lambda),\label{yyy27}\\
2((a-1)(\partial+\lambda)+a{\mu}-b)&g_1(\partial, \lambda)=(\partial+2\lambda+2\mu)g_1(-\lambda-\mu, \lambda),\label{yyy28}\\
2((a-1)(\partial+\lambda)+a{\mu}-b)&g_2(\partial, \lambda)=((a-1)\partial+a{(\lambda+\mu)}-b)g_2(-\lambda-\mu, \lambda)\nonumber\\
&\hphantom{g_2(\partial, \lambda)=((a-1)\partial}+((a-1)\partial+a{\mu}-b)f_1(\partial+\mu, \lambda),\label{yyy29}\\
(\partial+a(\lambda+\mu)+b)&g_1(-\lambda-\mu, \lambda)+((a-1)\partial+a{\mu}-b)g_1(\partial+\mu, \lambda)=0,\label{yyy30}\\
2(\partial+\lambda+2\mu)h_1(\partial, \lambda)&=(\partial+2\lambda+2\mu)h_1(-\lambda-\mu, \lambda)+(\partial+2\mu)h_1(\partial+\mu, \lambda),\label{yyy31}\\
2(\partial+\lambda+2\mu)h_2(\partial, \lambda)&=((a-1)\partial+a{(\lambda+\mu)}-b)h_2(-\lambda-\mu, \lambda)\nonumber\\
&\hphantom{=((a-1)\partial+a{(\lambda+\mu)}-b)} +(\partial+a\mu+b)h_2(\partial+\mu, \lambda),\label{yyy32}\\
2(\partial+\lambda+a\mu+b)q_1(\partial, \lambda)&=(\partial+2\mu)q_1(\partial+\mu, \lambda),\label{yyy33}\\
2(\partial+\lambda+a\mu+b)q_2(\partial, \lambda)&=(\partial+a(\lambda+\mu)+b)h_1(-\lambda-\mu, \lambda)+(\partial+a\mu+b)q_2(\partial+\mu, \lambda),\label{yyy34}\\
2((a-1)(\partial+\lambda)+a{\mu}-b)&q_1(\partial, \lambda)=(\partial+2\lambda+2\mu)q_1(-\lambda-\mu, \lambda),\label{yyy35}\\
2((a-1)(\partial+\lambda)+a{\mu}-b)&q_2(\partial, \lambda)=((a-1)\partial+a{(\lambda+\mu)}-b)q_2(-\lambda-\mu, \lambda)\nonumber\\
&\hphantom{q_2(\partial, \lambda)=((a-1)\partial} +((a-1)\partial+a{\mu}-b)h_1(\partial+\mu, \lambda),\label{yyy36}\\
(\partial+a(\lambda+\mu)+b)&q_1(-\lambda-\mu, \lambda)+((a-1)\partial+a{\mu}-b)q_1(\partial+\mu, \lambda)=0.\label{yyy37}
\end{align}

\begin{lemma}\label{lems1}
Depending on whether or not $a=2$, we have the following two cases:

(A1) If $a\neq2$, then
\begin{align*}
f_1(\partial, \lambda)=f_2(\partial, \lambda)=g_1(\partial, \lambda)=g_2(\partial, \lambda)=h_1(\partial, \lambda)=h_2(\partial, \lambda)=q_1(\partial, \lambda)=q_2(\partial, \lambda)=0;
\end{align*}

(A2) If $a=2$, then
\begin{align*}
&g_1(\partial, \lambda)=q_1(\partial, \lambda)=0,\quad
h_1(\partial, \lambda)=q_2(\partial, \lambda)=c_0,\\
&f_1(\partial, \lambda)=g_2(\partial, \lambda)=p(\lambda),\quad
f_2(\partial, \lambda)=s(\lambda),\quad
h_2(\partial, \lambda)=l(\lambda),
\end{align*}
where $p(\lambda),s(\lambda),l(\lambda)\in\mathbb{C}[\lambda]$ and $c_0\in\mathbb{C}$.
\end{lemma}

\begin{proof}
By comparing the terms of highest degree of $\lambda$ in Eqs.~\eqref{yyy26} and \eqref{yyy33}, respectively, we obtain $g_1(\partial, \lambda)=q_1(\partial, \lambda)=0$.
Moreover, Eq.~\eqref{yyy23} reduces to
\begin{align}
q_2(\partial+\lambda, \mu)q_2(\partial, \lambda)=q_2(-\lambda-\mu, \lambda)q_2(\partial, \lambda+\mu),
\end{align}
which implies the same equation as Eq.~\eqref{ly3}, so we have $q_2(\partial, \lambda)=c_0$ for some $c_0\in\mathbb{C}$.
Setting $\mu=0$ in Eq.~\eqref{yyy31}, we obtain $(\partial+2\lambda)(h_1(\partial, \lambda)-h_1(-\lambda, \lambda))=0$,
whence $h_1(\partial, \lambda)=h_1(-\lambda, \lambda)$. Similarly, setting $\mu=0$ in Eq.~\eqref{yyy34}, we get
\begin{align}\label{zmy1}
c_{0}(\partial+2\lambda+b)=(\partial+a\lambda+b)h_1(-\lambda, \lambda).
\end{align}
If $a=2$, then $(\partial+2\lambda+b)(h_1(-\lambda, \lambda)-c_0)=0$, hence $h_1(\partial, \lambda)=h_1(-\lambda, \lambda)=c_0$.
If $a\neq2$, by comparing the similar terms in Eq.~\eqref{zmy1}, one can obtain $h_1(\partial, \lambda)=c_0=0$.

Setting $\mu=0$ in Eq.~\eqref{yyy24}, we get $(\partial+2\lambda)(f_1(\partial, \lambda)-f_1(-\lambda, \lambda))=0$,
hence $f_1(\partial, \lambda)=f_1(-\lambda, \lambda)=p(\lambda)$ for some $p(\lambda)\in\mathbb{C}[\lambda]$.
Substituting this and $\mu=0$ into Eq.~\eqref{yyy27}, we obtain
\begin{align}\label{zmy2}
(\partial+2\lambda+b)g_2(\partial, \lambda)=(\partial+a\lambda+b)p(\lambda).
\end{align}
If $a=2$, then $(\partial+2\lambda+b)(g_2(\partial, \lambda)-p(\lambda))=0$, hence $g_2(\partial, \lambda)=p(\lambda)$.
If $a\neq2$, by comparing the terms of highest degree of $\partial$ in Eq.~\eqref{zmy2}, we get $g_2(\partial, \lambda)=p(\lambda)$.
It follows further that $(2-a)\lambda p(\lambda)=0$, leading to $p(\lambda)=0$.

Setting $\mu=0$ in Eq.~\eqref{yyy25}, we obtain
\begin{align}\label{zmy3}
(\partial+2\lambda-b)f_2(\partial, \lambda)=((a-1)\partial+a\lambda-b)f_2(-\lambda, \lambda).
\end{align}
If $a=2$, then $(\partial+2\lambda-b)(f_2(\partial, \lambda)-f_2(-\lambda, \lambda))=0$, hence $f_2(\partial, \lambda)=f_2(-\lambda, \lambda)=s(\lambda)$ for some $s(\lambda)\in\mathbb{C}[\lambda]$.
If $a\neq2$, by comparing the terms of highest degree of $\partial$ in Eq.~\eqref{zmy3}, one can get $f_2(\partial, \lambda)=f_2(-\lambda, \lambda)=0$.
In a similar way, setting $\mu=0$ in Eq.~\eqref{yyy32}, we have
\begin{align}\label{zmy4}
(\partial+2\lambda-b)h_2(\partial, \lambda)=((a-1)\partial+a\lambda-b)h_2(-\lambda, \lambda).
\end{align}
which implies the same equation as Eq.~\eqref{zmy3}. Therefore, when $a=2$, $h_2(\partial, \lambda)=h_2(-\lambda, \lambda)=l(\lambda)$, where $l(\lambda)\in\mathbb{C}[\lambda]$;
and when $a\neq2$, $h_2(\partial, \lambda)=h_2(-\lambda, \lambda)=0$.
This ends the proof.
\end{proof}

\begin{lemma}\label{lems2}
In C\textsc{ase} (A2) in Lemma \ref{lems1}, we have the following results:





(B1) $f_2(\partial, \lambda)=s(\lambda)$, $h_1(\partial, \lambda)=q_2(\partial, \lambda)=f_1(\partial, \lambda)=g_2(\partial, \lambda)=h_2(\partial, \lambda)=0$, where $s(\lambda)\in\mathbb{C}[\lambda]$;

(B2) $f_2(\partial, \lambda)=c_2\lambda$, $f_1(\partial, \lambda)=g_2(\partial, \lambda)=c_1$, $h_1(\partial, \lambda)=q_2(\partial, \lambda)=h_2(\partial, \lambda)=0$, where $c_2\in\mathbb{C}, c_1\in \mathbb{C}\backslash\{0\}$;

(B3) $f_2(\partial, \lambda)=c_3$, $f_1(\partial, \lambda)=g_2(\partial, \lambda)=h_2(\partial, \lambda)=c_1$, $h_1(\partial, \lambda)=q_2(\partial, \lambda)=0$, where $c_3\in\mathbb{C}, c_1\in \mathbb{C}\backslash\{0\}$;

(B4) $f_2(\partial, \lambda)=\frac{km}{c_0}\lambda-\frac{k^2}{c_0}\lambda^2$, $f_1(\partial, \lambda)=g_2(\partial, \lambda)=m-k\lambda$, $h_2(\partial, \lambda)=k\lambda$, $h_1(\partial, \lambda)=q_2(\partial, \lambda)=c_0$, where $m,k\in\mathbb{C}, c_0\in \mathbb{C}\backslash\{0\}$.

\end{lemma}

\begin{proof}
In C\textsc{ase} (A2) in Lemma \ref{lems1}, Eqs.~\eqref{yyy8}--\eqref{yyy37} are equivalent to the following equations:
\begin{align}
&c_{0}l(\lambda)+c_{0}l(\mu)=c_{0}l(\lambda+\mu),\label{yyy43}\\
&c_{0}s(\mu)+p(\mu)l(\lambda)=c_{0}s(\lambda+\mu)+l(\lambda)l(\lambda+\mu),\label{yyy44}\\
&c_{0}l(\lambda)+c_{0}p(\lambda+\mu)=c_{0}p(\mu),\label{yyy45}\\
&p(\lambda)l(\mu)+c_{0}s(\lambda)=p(\lambda)l(\lambda+\mu),\label{yyy46}\\
&p(\lambda)s(\mu)+p(\mu)s(\lambda)=p(\lambda)s(\lambda+\mu)+s(\lambda)l(\lambda+\mu),\label{yyy47}\\
&p(\lambda)p(\lambda+\mu)+c_{0}s(\lambda)=p(\lambda)p(\mu)\label{yyy48}.
\end{align}
To solve the Eqs.~\eqref{yyy43}--\eqref{yyy48}, we distinguish the following two cases:

C\textsc{ase} (I): $c_{0}=0$. Then it follows that Eqs.~\eqref{yyy43} and \eqref{yyy45} hold. By Eq.~\eqref{yyy48}, we get
$p(\lambda)p(\lambda+\mu)=p(\lambda)p(\mu)$, which implies $p(\lambda)=c_1$ for some $c_1\in \mathbb{C}$.

S\textsc{ubcase} (IA): $c_{1}=0$. By Eq.~\eqref{yyy44}, we get $l(\lambda)l(\lambda+\mu)=0$, hence $l(\lambda)=0$.
Consequently, one can readily verify that Eqs.~\eqref{yyy46} and \eqref{yyy47} hold for all $s(\lambda)\in \mathbb{C}[\lambda]$.

S\textsc{ubcase} (IB): $c_{1}\neq0$. Combining Eqs.~\eqref{yyy44} and \eqref{yyy46}, we get $l(\lambda)(l(\lambda)-c_1)=0$, which implies
$l(\lambda)=0$ or $l(\lambda)=c_1$. If $l(\lambda)=0$, by Eq.~\eqref{yyy47}, we obtain $s(\mu)+s(\lambda)=s(\lambda+\mu)$.
It follows that the degree of $s(\lambda)$ is at most one. Hence we may write $s(\lambda)=c_4+c_2\lambda$ with $c_4,c_2\in \mathbb{C}$,
taking it into the equation above, we get $c_4=0$. Thus, $s(\lambda)=c_2\lambda$ for some $c_2\in \mathbb{C}$.
If $l(\lambda)=c_1$, by Eq.~\eqref{yyy47}, we get $s(\mu)=s(\lambda+\mu)$, leading to $s(\lambda)=c_3$ for some $c_3\in \mathbb{C}$.

C\textsc{ase} (II): $c_{0}\neq0$. By Eq.~\eqref{yyy43}, we obtain $l(\lambda)+l(\mu)=l(\lambda+\mu)$. It follows that $l(\lambda)=k\lambda$ for some $k\in \mathbb{C}$.
Substituting it into Eq.~\eqref{yyy45}, we get $p(\lambda+\mu)-p(\mu)=-k\lambda$.
It is obvious that the degree of $p(\lambda)$ is at most one. Setting $p(\lambda)=m+k_1\lambda$, where $m,k_1\in \mathbb{C}$,
it follows further that $k_1=-k$. Hence $p(\lambda)=m-k\lambda$ for some $m,k\in \mathbb{C}$. 
Combining Eqs.~\eqref{yyy43} and \eqref{yyy46}, we get $c_0s(\lambda)=p(\lambda)l(\lambda)$. Therefore, $s(\lambda)=\frac{km}{c_0}\lambda-\frac{k^2}{c_0}\lambda^2$, where $m,k\in \mathbb{C},c_0\in\mathbb{C}\backslash\{0\}$.
Furthermore, it is easy to check that Eqs.~\eqref{yyy44}, \eqref{yyy47}, and \eqref{yyy48} are satisfied.

The proof is completed.
\end{proof}

\begin{theorem}\label{theorem1}
The compatible noncommutative TPCA structures over the Lie conformal algebra $W(a,b)$ are of the forms:

(1) If $a\neq 2$, then
\begin{align*}
L\circ_\lambda L=L\circ_\lambda M=M\circ_\lambda L=M\circ_\lambda M=0;
\end{align*}

(2) If $a=2$, then

\indent\indent(2.1) $L\circ_\lambda L= s(\lambda)M,\quad L\circ_\lambda M=M\circ_\lambda L=M\circ_\lambda M=0$;  

\indent\indent(2.2) $L\circ_\lambda L= c_1L+c_2\lambda M,\quad L\circ_\lambda M=c_1M, \quad M\circ_\lambda L=M\circ_\lambda M=0$; 

\indent\indent(2.3) $L\circ_\lambda L= c_1L+c_3M,\quad L\circ_\lambda M=M\circ_\lambda L=c_1M, \quad M\circ_\lambda M=0$;

\indent\indent(2.4) $L\circ_\lambda L= (m-k\lambda)L+(\frac{km}{c_0}\lambda-\frac{k^2}{c_0}\lambda^2)M,\quad L\circ_\lambda M=(m-k\lambda)M,$\\
\hphantom{\indent\indent(2.4)} $M\circ_\lambda L=c_0L+k\lambda M, \quad M\circ_\lambda M=c_0M$;

where $s(\lambda)\in \mathbb{C} [\lambda]$, $c_0,c_1\in \mathbb{C}\backslash\{0\}$, $c_2,c_3,m,k\in\mathbb{C}$.
\end{theorem}
\begin{proof}
This theorem follows directly from Lemma \ref{lems1} and Lemma \ref{lems2}.
\end{proof}

In order to classify all non-isomorphic noncommutative TPCA structures with the underlying Lie conformal algebra isomorphic to
$W(2,b)$, we discuss each of the cases {(2.1)--(2.4)} in Theorem \ref{theorem1} and present the nonzero $\lambda$-products.

In C\textsc{ase} (2.1), providing that $0 \neq s(\lambda)\in \mathbb{C} [\lambda]$, the compatible noncommutative TPCA structures on $W(2,b)$ are nontrivial.

In C\textsc{ase} (2.2), since $c_1\neq0$, then after the changing of the $\mathbb{C}[\partial]$-basis $[L':=L-\frac{c_2}{c_1}(\partial-b)M,\, M':= M]$, we obtain
\begin{align*}
L'\circ_\lambda L'=c_1L',\quad L'\circ_\lambda M'=c_1M',\quad M'\circ_\lambda L'=M'\circ_\lambda M'=0.
\end{align*}

In C\textsc{ase} (2.3), if $b\neq0$, then after the changing of the $\mathbb{C}[\partial]$-basis $[L':=L+\frac{c_3}{bc_1}(\partial-b)M,\, M':= M]$, we obtain
\begin{align*}
L'\circ_\lambda L'=c_1L',\quad L'\circ_\lambda M'=M'\circ_\lambda L'=c_1M',\quad M'\circ_\lambda M'=0.
\end{align*}
As for $b=0$, we consider the following two subcases.

S\textsc{ubcase} (I): In the case $c_{3}=0$, we directly obtain
\begin{align*}
L\circ_\lambda L=c_1L,\quad L\circ_\lambda M=M\circ_\lambda L=c_1M,\quad M\circ_\lambda M=0.
\end{align*}

S\textsc{ubcase} (II): In the case $c_{3}\neq0$, after the changing of the $\mathbb{C}[\partial]$-basis $[L':=L,\, M':= c_3M]$, we have
\begin{align*}
L'\circ_\lambda L'=c_1L'+M',\quad L'\circ_\lambda M'=M'\circ_\lambda L'=c_1M',\quad M'\circ_\lambda M'=0.
\end{align*}

In C\textsc{ase} (2.4), since $c_0\neq0$, then after the changing of the $\mathbb{C}[\partial]$-basis $[L':=L,\, M':=\frac{1}{c_0} M]$, we have
\begin{align*}
&L'\circ_\lambda L'= (m-k\lambda)L'+(km\lambda-{k^2}\lambda^2)M',\\
&L'\circ_\lambda M'=(m-k\lambda)M',\quad M'\circ_\lambda L'=L'+k\lambda M', \quad M'\circ_\lambda M'=M'.
\end{align*}
Next we make the transformation $[L:=L'-k(\partial-b)M',\, M:= M']$ and obtain
\begin{align*}
L\circ_\lambda L=c L,\quad L\circ_\lambda M=c M,\quad M\circ_\lambda L=L,\quad M\circ_\lambda M=M,\quad where\,\, c=m+kb.
\end{align*}

We summarize the discussions above in the following theorem:
\begin{theorem}
Let $(W(a,b),\circ_\lambda,[\cdot_\lambda\cdot])$ be a nontrivial noncommutative TPCA structures over the Lie conformal algebra $W(a,b)$. Then $(W(a,b),\circ_\lambda,[\cdot_\lambda\cdot])$
is isomorphic to one of the following:
\begin{enumerate}
\item[{(1)}] $L\circ_\lambda L= s(\lambda)M,\quad L\circ_\lambda M=M\circ_\lambda L=M\circ_\lambda M=0$;
\item[{(2)}] $L\circ_\lambda L=c_1L,\quad L\circ_\lambda M=c_1M,\quad M\circ_\lambda L=M\circ_\lambda M=0$;
\item[{(3)}] $L\circ_\lambda L=c_1L,\quad L\circ_\lambda M=M\circ_\lambda L=c_1M,\quad M\circ_\lambda M=0$;
\item[{(4)}] $L\circ_\lambda L=c_1L+M,\quad L\circ_\lambda M=M\circ_\lambda L=c_1M,\quad M\circ_\lambda M=0$;
\item[{(5)}] $L\circ_\lambda L=c L,\quad L\circ_\lambda M=c M,\quad M\circ_\lambda L=L,\quad M\circ_\lambda M=M$;
\end{enumerate}
where $0 \neq s(\lambda)\in \mathbb{C} [\lambda]$, $c_1\in \mathbb{C}\backslash\{0\}$, $c\in\mathbb{C}$.
\end{theorem}




\bibliography{tpca}
\bibliographystyle{plainurl}

\end{document}